\documentclass[12pt]{amsart}
\usepackage{amsmath}
\usepackage{amssymb}
\usepackage{amsthm}
\usepackage{epsfig,graphicx}
\usepackage[arrow,matrix,tips,2cell]{xy}

\newcommand{\ignore}[1]{\relax}

\numberwithin{equation}{section}

\newcommand{\R}{\mathbb R}
\newcommand{\Z}{\mathbb Z}
\newcommand{\N}{\mathbb N}

\newcommand{\Q}{\mathbb Q}
\newcommand{\T}{\mathbb T}
\newcommand{\U}{\mathcal U}

\newcommand{\F}{\mathcal F}
\newcommand{\W}{\mathcal W}

\newcommand{\s}{\mathcal E}

\newcommand{\ostar}{\operatorname{St}}
\newcommand{\Sk}{\operatorname{Sk}}

\newcommand{\Aut}{\operatorname{Aut}}
\newcommand{\GL}{\operatorname{GL}}

\newtheorem{lem}{Lemma}[section]

\newtheorem{theorem}[lem]{Theorem}
\newtheorem{lemma}[lem]{Lemma}
\newtheorem{corollary}[lem]{Corollary}

\newtheorem{conjecture}[lem]{Conjecture}

\newtheorem{proposition}[lem]{Proposition}

\theoremstyle{definition}
\newtheorem{condition}{Condition}

\newtheorem{definition}[lem]{Definition}

\newtheorem{example}[lem]{Example}

\theoremstyle{remark}
\newtheorem{remark}[lem]{Remark}

\newtheorem*{acknowledgement}{Acknowledgements}

\newcommand{\dd}{\partial}

\newcommand{\tp}{{\mathbb T}{\mathbb P}}

\newcommand{\Vol}{\operatorname{Vol}}

\renewcommand{\setminus}{\smallsetminus}

\newcommand{\OO}{\mathcal O}

\newcommand{\Hom}{\operatorname{Hom}}

\newcommand{\<}{\langle}   
\renewcommand{\>}{\rangle} 

\newcommand{\Link}{\operatorname{Link}}

\begin{document}

\title{Tropical eigenwave and intermediate Jacobians}
\author{Grigory Mikhalkin}
\address{Universit\'e de Gen\`eve,  Math\'ematiques, Villa Battelle, 1227 Carouge, Suisse}
\email{grigory.mikhalkin@unige.ch}
\author{Ilia Zharkov}
\address{Kansas State University, 138 Cardwell Hall, Manhattan, KS 66506 USA}
\email{zharkov@math.ksu.edu}

\begin{abstract}
Tropical manifolds are polyhedral complexes enhanced with certain kind of affine structure.
This structure manifests itself through a particular cohomology class which we call the eigenwave
of a tropical manifold. Other wave classes of similar type are responsible for deformations of the
tropical structure. 

If a tropical manifold is approximable by a 1-parametric family of complex manifolds then the
eigenwave records the monodromy of the family around the tropical limit. With the help of tropical
homology and the eigenwave we define tropical intermediate Jacobians which can be viewed
as tropical analogs of classical intermediate Jacobians.
\end{abstract}
\thanks{
Research is supported in part
by the NSF FRG grant DMS-0854989 (G.M. and I.Z.),
the TROPGEO project of the European Research Council (G.M.) and
the Swiss National Science Foundation  grants 140666 and 141329 (G.M.).
}

\maketitle

\section{Tropical spaces and tropical manifolds}
In this section we briefly recall basic concepts of tropical spaces relevant for our paper.  For more details we refer to \cite{Mik06} and \cite{MR}.
The main assumption we make is that our the tropical space is regular at infinity.

\subsection{Tropical spaces}
A tropical affine $n$-space $\T^n$ is the topological space $[-\infty,\infty)^n$ (homeomorphic to the
$n$th power of a half-open interval) enhanced with a collection of functions $\OO_{\operatorname{pre}}=\{f\}$, $f:U\to\T=[-\infty,\infty)$. 
Here $U\subset\T^n$ is an open set and $f$ is a function that can be expressed as
\begin{equation}\label{fx}
f(x)=\max_{j\in A} (jx + a_j)
\end{equation}
for a finite set $A\subset\Z^n$ and a collection of numbers $a_j\in\T$,
such that the scalar product $jx$ is well-defined as a number in $\T$ (i.e. is finite or $-\infty$) for any $x\in U$.

The collection of functions $\OO_{\operatorname{pre}}$ is a presheaf which gives rise to a sheaf $\OO$ of {\em regular functions} on $\T^n$ (which we will also denote $\OO_{\T^n}$
indicating the space where it is defined to avoid ambiguity). $\OO$ is called the {\em structure sheaf} on $\T^n$.

It is convenient to stratify the space $\T^n$ by 
$$\T^\circ_I:= \{y\in \T^n \ :\ y_i=-\infty, i\in I \ \text{ and } \ y_i>-\infty, i\notin I \},
$$ where $I\subset \{1,\dots,n\}$.
Each $T^\circ_I$ is isomorphic to $\R^{n-|I|}$ and we set $\T_I$ to be its closure in $\T^n$.

To write down a regular function \eqref{fx} on $\R^n$ all we need
is the {\em integral affine structure} on $\R^n$. This allows us to distinguish functions $\R^n\to\R$ which are affine with linear parts defined over $\Z$. 
Thus the tropical structure on $\T^n$ can be thought of as an extension
of the integral affine structure in $\R^n$ where the overlapping maps are compositions of linear transformations in $\R^n$ defined over $\Z$ with arbitrary translations in $\R^n$.

Given a subset $U\subset\T^N$  we say that
a {\em continuous} map $U\to\T^M$ is {\em integral affine}
if it restricts to an affine map $\R^N\to \R^M$ with integral linear part.
We say that a partially defined map
$h:\T^{N}  \dashrightarrow \T^{M}$ is {integral affine}
if it is defined on a subset $U\supset\R^N$ and is integral affine there.
Extending $h$ whenever we can by continuity we see that for each $I\subset\{1,\dots,N\}$
$h$ is defined everywhere or nowhere on $\T^\circ_I$.

The automorphisms of a subset $U\subset\T^N$ are invertible integral affine maps $U\to U$.
For example, the automorphisms $\Aut (\R^N) \cong \GL_N (\Z) \ltimes \R^N$
 form a group of all integral affine transformations of $\R^N$
while $\Aut (\T^{N})\cong \R^N$ only consists of translations.
We also note that automorphisms of  $\T^{s} \times \R^{N-s}$
translate an $s$-dimensional affine subspace of $\R^N$ parallel to the $\T^{s}$ factor
to another one with the same property.

A {\em convex polyhedral domain} $D$ in $\T^{N}$ is defined as the intersection
of a finite collection of half-spaces $H_k$ of the form
\begin{equation}\label{H_k}
H_k=\{x\in\T^N\ |\ jx\le a\}\subset\T^{N}
\end{equation}
for some $j\in\Z^N$ and $a\in\R$. 
The boundary $\dd H_k$ is given by the equation $jx=a$. 
A {\em mobile face} $E$ of $D$ is the intersection
of $D$ with the boundaries of some of its defining half-spaces given by \eqref{H_k}.
The adjective {\em mobile} stands here to distinguish such faces among
more general faces of $X$ which we will define later and 
which are allowed to have support in $\T^N\setminus\R^N$, i.e. be disjoint from
$\R^N\subset \T^{N}$. (They have reduced mobility and are called {\em sedentary}).

The {\em dimension} of a convex polyhedral domain $D$ is its topological dimension. 
Observe that for each mobile face $E$ of $D$ the intersection 
$$E^\circ=E\cap\R^N$$
is non-empty. The intersection $E^\circ$ is called the non-infinite part of a mobile face.
Each mobile face of $D$ is a convex polyhedral domain itself (although perhaps of smaller dimension).

We say two domains $D\subset \T^{N}$ and $D'\subset \T^{M}$ are isomorphic if there is an integral affine map $\T^{N} \dashrightarrow \T^{M}$ which restricts to a homeomorphism $D\to D'$
(in particular, it has to be defined everywhere on $D$).

We say that a convex polyhedral domain $D\subset \T^{N}$ is {\em  regular at infinity}
if for every $I\subset \{1,\dots, N\}$ the intersection $D\cap (\T^{\circ}_I)$ is either empty or is
a $(\dim D-|I|)$-dimensional  polyhedral domain in $\T^\circ_I\cong \R^{N-|I|}$. 

\begin{definition}\label{def:poly_complex}
An $n$-dimensional polyhedral complex $Y=\bigcup D\subset \T^{N}$
is the union of a finite collection of convex $n$-dimensional polyhedral domains $D$, called
the {\em facets} of $Y$ subject to the following property.
For any collection $\{D_j\}$ of facets, their intersection $\bigcap D_j$ is a face of each $D_j$.
Such intersections are called the (mobile) faces of $Y$. Clearly they are themselves polyhedral domains in $\T^{N}$.

We say that $Y$ is regular at infinity if all its faces are regular at infinity.
\end{definition}

In this paper we assume that all polyhedral complexes are regular at infinity.

\begin{condition}[Balancing]
Let $E$ be an $(n-1)$-dimensional mobile face in $Y$
and $D_1,\dots,D_l\subset \T^{N}$ be the facets adjacent to $E$. 
Take the quotient of $\R^{N}$ by the linear subspace parallel to $E^\circ$,
the non-infinite part of $E$.
The balancing condition requires that
\begin{equation}\label{bal-cond}
\sum\limits_{k=1}^l\epsilon_k=0,
\end{equation}
where the $\epsilon_k$
are the outward primitive integer vectors parallel to the images of $D_k$ in this quotient.
\end{condition}

A polyhedral complex $Y\subset\T^N$ is called {\em balanced} if all of its $(n-1)$-dimensional faces satisfy the balancing condition.

More generally we can consider  spaces that locally look like balanced polyhedral
complexes, i.e. admit a covering by open sets
$U_\alpha$ enhanced with open embeddings (charts)
$$\phi_\alpha:U_\alpha\to Y_\alpha\subset\T^{N_\alpha}
$$
where each $Y_\alpha\subset\T^{N_\alpha}$ is a balanced polyhedral complex.
In this paper we assume in addition that each $Y_\alpha$ is regular at infinity.

We may express compatibility of different charts by requiring that the corresponding overlapping maps are induced by integral affine maps $\T^{N_\alpha}\dashrightarrow\T^{N_\beta}$.  Or, equivalently, we may use the structure sheaf and enhance each
$Y_\alpha\subset\T^{N_\alpha}$ with the sheaf $\OO_{Y_\alpha}$ induced from $\OO_{\T^{N_\alpha}}$.
Its pull-back under $\phi_\alpha$ is a sheaf on $U_\alpha$. Two charts $\phi_\alpha$ and
$\phi_\beta$ are compatible if the corresponding restrictions to $U_\alpha\cap U_\beta$
agree.

We arrive to the following definition of a tropical space.
\begin{definition}[cf. \cite{MR}]\label{def-tropspace}
A tropical space 
is a topological space $X$ enhanced with a cover of compatible charts
$\phi_\alpha:U_\alpha\to  Y_{\alpha}\subset\T^{N_\alpha}$
to balanced polyhedral complexes as above and which satisfies the finite type condition below.

The tropical space $X$ is regular at infinity if it admits charts to polyhedral complexes regular at infinity.
\end{definition}
The  charts induce a sheaf $\OO_X$ on $X$ which we call the structure sheaf of $X$.

\begin{condition}[Finite type] The number of charts $\phi_\alpha$ covering $X$ is finite
while each chart is subject to the following property.
If $\{x_j\in U_\alpha\}_{j=1}^\infty$ is a sequence such that $\phi_\alpha(x_j)$ converges
to a point $y\in\T^{N_\alpha}$
then either the sequence $\{x_j\}$ converges inside the topological space $X$ 
or there exists a coordinate in $\T^{N_\alpha}$ such that its value on $y$ is $-\infty$ while
its value on any point in $\phi_\alpha(U_\alpha)$ is finite.
\end{condition}

It is easy to see that this finite type condition is a reformulation of the one from \cite{MR}.

\subsection{Sedentary points and faces}
Let $D \subset \T^{N}$ be a polyhedral domain.
It is convenient  to treat 
the intersections $D \cap\T_I$ for $I\subset\{1,\dots,N\}$
also as its faces (at infinity). If we need to distinguish such faces from the mobile ones we have defined
before we call these new faces {\em sedentary}.
\begin{definition}\label{def:sedentarity}
We say that
$$E_I:=E \cap\T_I$$ is a face of $D$ if $E$ is a mobile face of $D$.
The {\em sedentarity} of the face $E_I$ is $s=|I|$,
while its {\em refined sedentarity} is $I$. 
\end{definition}

Clearly, the {mobile} faces (defined previously) are the faces of sedentarity $0$.
If $Y\subset \T^{M}$ is a polyhedral complex then we define a (possibly sedentary)
face of $Y$ as a face of a facet in $Y$. 

We will use the notation $F\prec^s_j E$ when $F$ is a face of $E$ of codimension $j$ and sedentarity $s$ higher. It is also convenient to introduce the following terminology.
\begin{definition}
A face $E$ of $Y$ is called {\em infinite}
if either it is not compact or it contains a higher sedentary subface.
Otherwise $E$ is called finite (even if the sedentarity of $E$ itself is positive).
\end{definition}

Note that even though a  face $F\subset Y$ of sedentarity $I$
may be adjacent to several facets, it is always presented as
$$F=E\cap\T^I$$
for a unique mobile face $E\subset Y$ which we call the {\em parent} of $F$
(as long as $Y$ is regular at infinity). The set of faces of $Y$ with the same parent $E$ is called the {\em family} of $E$. In case $E$ is compact the regularity at infinity forces its family to have a very simple combinatorial structure.

\begin{proposition}\label{prop:family}
Let $E\subset Y$ be a compact mobile face containing a face of a maximal sedentarity $s$. Then its family $\Pi(E)$ forms a lattice poset (under $\prec_j^j$), isomorphic to the face poset of a simplicial cone of dimension $s$. The maximal sedentary face in the poset is finite. 
\end{proposition}

Note also that a face $F$ of sedentarity $I$ completely determines the integral affine structure of its parent face $E$ in the neighborhood of $\T_I$.
Namely, we have the following proposition.

\begin{proposition}
 Let $\pi_I:\T^N\to\T^I$ be the projection
taking a point $(x_1,\dots,x_N)$  to the point whose $j$-th coordinate is $x_j$ if  $j\notin I$
and $-\infty$ otherwise.
The parent face $E$ of $F$ is contained in $\pi_I^{-1}(F)$. Furthermore, for a small
open neighborhood $U\supset \T_I$ we have $$E\cap U= \pi_I^{-1}(F)\cap U.$$
\end{proposition}

In other words for a sufficiently small $\epsilon>-\infty$
we have $(x_1,\dots,x_N)\in E$ whenever $\pi_I(x_1,\dots,x_N)\in F$
and $x_j<\epsilon$ for any $j\in I$.
Thus the directions parallel to the $j$-th coordinate in $\T^N$ for $j\in I$ are quite special for $E$. 
We orient them toward the $-\infty$-value of the coordinate and
call them {\em divisorial directions},  see Figure \ref{mobile}.
Their positive linear combinations span the {\em divisorial cone} while all
linear combination span the {\em divisorial subspace} in $\R^N$. 
The primitive integral vector along a divisorial direction
(pointing towards $-\infty$ as the direction itself) is called a {\em divisorial vector}.
\begin{figure}
\centering
\includegraphics[height=45mm]{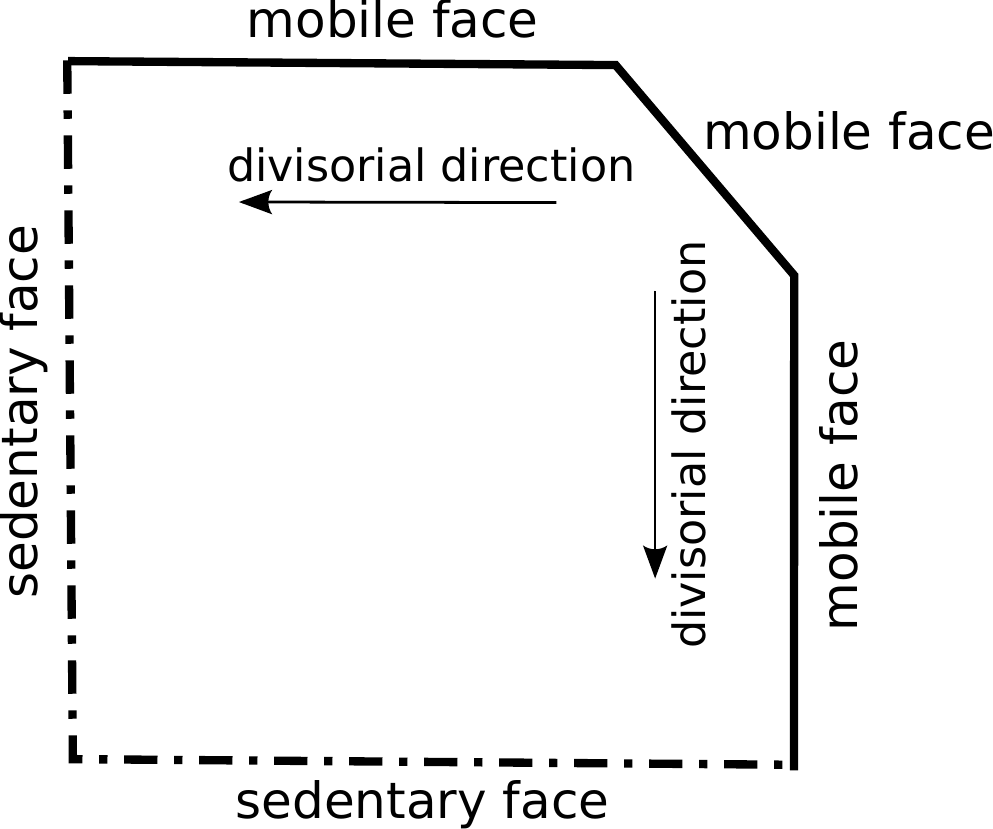}
\caption{Mobile and sedentary faces of a polyhedral domain in $\T^N$.}
\label{mobile}
\end{figure}

One important observation is that the divisorial vectors are invariant with respect to any integral affine automorphism of $\T^{|I|}\times\T_I^\circ$. Thus they are intrinsically defined for $F$
and so is the divisorial subspace which we denote by $W^{div}$. 

\subsection{Tangent spaces}\label{sec:tangent}
Let $y$ be a point in the relative interior of a face $F$ of sedentarity $I$ in a balanced polyhedral complex $Y\subset\T^N$. Let $\Sigma(y)$ be the cone in $\T^\circ_I\cong\R^{N-|I|}$ consisting of vectors $u\in\T^\circ_I$ such that $y+\epsilon u\in Y\cap \T^\circ_I$ for a sufficiently small $\epsilon>0$ (depending on $u$). We denote the intersection of all maximal linear subspaces contained in $\Sigma(y)$ by~$W'(y)$.

Clearly, the cones $\Sigma(y)$ can be canonically identified for all points $y$ in the relative interior of the same face, and so can be the vector spaces $W'(y)$.
We say that $y_m\in Y$ is a {\em nearby mobile point} to $y$ if $y_m$ belongs to the relative interior of the parent face to $F$.
\begin{definition}
For a point $y\in Y$ we define $W(y)$, the {\em wave tangent space} at $y$, as $W'(y_m)$ for a nearby mobile point $y_m$. The (conventional) tangent space $T(y)$ at $y$ is defined as the linear span of $\Sigma(y)$ in $\T^\circ_I\cong\R^{N-|I|}$, where $I$ is the refined sedentarity of $y$.
\end{definition}

Note that there are two essential distinctions in defining  $T(y)$ and $W(y)$.
To define $W(y)$ we always move to a nearby mobile point $y_m$. The space $W'(y)$ itself, is naturally a quotient of $W(y)$ by the divisorial subspace $W^{div}(y)$.

On the other hand, for $T(y)$ we work in a vector space $\T^\circ_I$, which is naturally the quotient $\R^N/W^{div}(y)$,
but we take the linear span of the cone instead of the vector space contained in it.

If we need to specify the space $Y$ for the tangent space $T(y)$ we write $T_Y(y)$, and similarly for $W(y)$.
The following proposition is straightforward.
\begin{proposition}\label{prop:differentials}
An integral affine map $h:\T^N\dashrightarrow\T^M$ induces linear maps
$dh^W: W_Y(y)\to W_{h(Y)}(h(y))$ and $dh^T: T_Y(y)\to T_{h(Y)}(h(y))$ whenever $h$ is well-defined on~$y$.
We call these maps differentials of $h$.

The differentials are natural in the following sense. If $g :\T^M\dashrightarrow\T^L$ is another integral affine map defined on $h(y)$, then the induced differentials satisfy $d(g\circ h)=(dg)\circ (dh)$.
\end{proposition}

Let $x\in X$ be now a point in a tropical space.
\begin{corollary}
The tangent spaces $W_{Y_\alpha}(\phi_\alpha(x))$ (resp. $T_{Y_\alpha}(\phi_\alpha(x))$)
for different charts  $\phi_\alpha$ are 
identified by the differentials of the overlapping maps.
The resulting spaces $W(x)$ and $T(x)$ are called the wave tangent space and
the (conventional) tangent space to the tropical space $X$ at its point $x$.
\end{corollary}

The tangent spaces $T(x)$ and $W(x)$ carry natural integral structure.
We denote the corresponding lattices by $T_\Z(x)$ and $W_\Z(x)$.

\subsection{Polyhedral structures}\label{sec:polyhedral}
Sometimes a tropical space $X$ comes with a structure of an (abstract) polyhedral
complex, which is not always the case.
\begin{definition}
We say that a tropical space $X$ is {\em polyhedral}
if there are finitely many closed subsets $\Delta_j\in X$ (called {\em facets}) with the following properties.
\begin{itemize}
\item For each $\Delta_j$ there exists a chart such that $\Delta_j\subset U_\alpha$ and
$\phi_\alpha(\Delta_j)$ is a facet of the balanced polyhedral complex $Y_\alpha\subset\T^{N_\alpha}$.
\item 
For any collection $\{\Delta_j\}$ of facets of $X$ and any face $\Delta_j$ in this collection
the intersection $\bigcap \Delta_j$ is a face of $\Delta_j$.
\end{itemize}
\end{definition}

Note that we may work with tropical polyhedral spaces in the same way as we work with 
balanced polyhedral complexes in $\T^N$. In particular, we can define in the same way their faces (which will denote by $\Delta$), both mobile and sedentary, parent faces with their families, divisorial directions, and any other notion which is intrinsically defined, that is stable under allowed integral affine maps. For instance, Proposition \ref{prop:family} will read:

\begin{proposition}\label{prop:subface}
Let $X$ be a compact polyhedral tropical space. 
For every face $\Delta$ of sedentarity $s$ there is a unique (parent) face $\Delta_0$ of sedentarity 0 such that $\Delta\prec^s_s\Delta_0$. The cells of $X$ with the same $\Delta_0$, the family of $\Delta_0$, form a lattice poset $\Pi(\Delta_0)$ isomorphic to the face poset of a simplicial cone. 
Every face of $X$ belongs to exactly one family poset $\Pi$. The maximal sedentary face $\Delta_{\min}$ in a poset is finite. 
\end{proposition}

We will denote the $k$-skeleton of  a polyhedral tropical space X (that is the union of $(\le k)$-dimensional faces) by $\Sk_k(X)$. It is often convenient to take the covering $\{U_\alpha\}$ by open stars of vertices. That is, each $U_v$ is the union of relative interiors of faces of $X$ adjacent to the vertex $v$. Then the relative interior of a face $\Delta$ is contained in every $U_v$ if $v$ is a vertex of $\Delta$.

Another useful feature of a compact polyhedral tropical space is that we can define its first baricentric subdivision. For a finite cell we take an arbitrary point in its interior for its baricenter.
\begin{figure}
  \centering
\includegraphics[height=30mm]{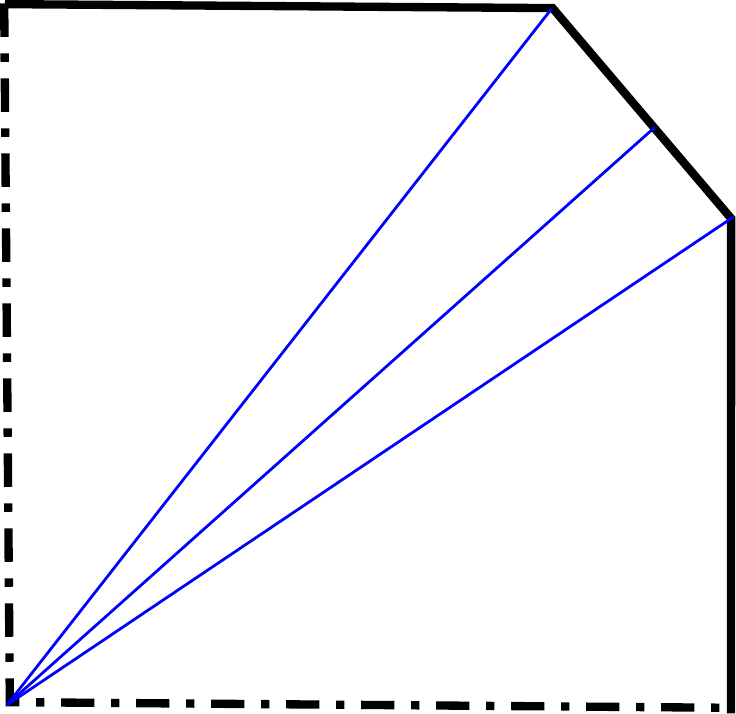}
\caption{\label{fig:bar} Baricentric subdivision of an infinite cell. The dotted faces have higher sedentarity.}
\end{figure}
For an infinite cell we take for its baricenter the baricenter of its unique most sedentary (necessarily finite, cf. Proposition \ref{prop:subface}) subface (see Figure \ref{fig:bar}).
That is, we first choose baricenters of maximal sedentary faces and then name them also as baricenters of any adjacent faces of lower sedentarity. The subdivision of each face of $X$ into simplices is constructed as usual by the flags of its subfaces of {\em minimal} sedentarity.

The baricentric subdivision of $X$ is {\em not} a polyhedral tropical space as we defined it. It violates the regularity at infinity property. Nevertheless, it is very convenient to have a triangulation of $X$. This enables us to define simplicial versions of the (co)homology theories which are very useful for carrying out explicit calculations.

\subsection{Combinatorial stratification}\label{sec:stratification}
Notice that a polyhedral structure on a tropical space (if it exists) is in no way unique.
In this subsection we define a combinatorial stratification which is not always polyhedral,
but is naturally defined on any tropical space $X$.

\begin{definition}
We say that two points $x,x'\in X$ are {\em combinatorially equivalent} 
if there exists a path connecting $x$ to $x'$ along which both 
the dimension of the wave tangent space $W$ and the sedentarity remain constant.
{\em A combinatorial stratum} of the tropical space $X$ is a class of combinatorial equivalence.
\end{definition}

We will denote combinatorial strata of $X$ by $\s$ and use the notation $\s\prec \s'$ if the stratum $\s$ lies on the boundary of $\s'$. 

\begin{example}\label{eg:elliptic} 
Consider the circle $E_l$ of length $l$,  otherwise called a {\em tropical elliptic curve}.
$E_l$ is a tropical space: we can present it as a tropical polyhedral space by
choosing, e.g., three distinct points so that they split $E_l$ into three facets.
This subdivision is not unique as we can move these points around or consider a subdivision into a larger number of facets.
The combinatorial stratification for $E_l$ is trivial: it consists of a single stratum $E_l$. 
\end{example}

Let two points $x,y \in U_\alpha\subset X$ belong to one chart $\phi: U_\alpha \to Y$ of $X$  and they sit  in some strata $x\in \s_x$ and $y\in \s_y$. If $\s_x=\s_y$, that is if they belong to the same stratum, one can canonically identify the tangent spaces $T (y) = T (x)$ and $W(x) = W(y)$. The identification is natural in the following sense. If the points also belong to another common covering open subset $U_\beta$ it commutes with the differentials induced by the overlapping map. 

In other words, we get flat connections on the bundles $T$ and $W$ over each combinatorial stratum of $X$.

Furthermore, if $\s_x\prec \s_y$ then one has two natural maps
\begin{equation}\label{eq:strata_maps}
\iota: T_\Z (y) \to T_\Z (x)  \text{ and } \pi :  W_\Z (x) \to W_\Z(y),
\end{equation}
(note the different directions) defined as follows. 
If $I(\phi(y))=I(\phi(x))$ then any face adjacent to $\phi(x)$ is contained in some face adjacent to $\phi(y)$ and $\iota$ is given by inclusion. If $I(\phi(y))\neq I(\phi(x))$ (note that we must
have $I(\phi(y))\subset I(\phi(x))$) then $\iota$ is the projection along the divisorial directions indexed by $I(\phi(x))\setminus I(\phi(y))$. The map $\pi$ is given by inclusion of the linear spaces spanned by the corresponding parent faces.

Again the maps $\iota$ and $\pi$  are natural in the sense that they commute with the overlapping differentials.

\subsection{Tropical manifolds}
First we recall a construction of a balanced polyhedral fan associated to a matroid
(\cite{AK}, see also e.g. \cite{Shaw}, \cite{MR}).

A matroid $M=(M,r)$ is a finite set $M$ together with a rank function $r:2^M\to\Z_{\ge0}$ such that
we have the inequalities
$r(A\cup B)+r(A\cap B) \le r(A) + r(B)$ and $r(A)\le |A|$,
where $|A|$ is the number of elements in $A$,
for any subsets $A,B\subset M$
as well as the inequality $r(A)\le r(B)$ whenever $A\subset B$.
Subsets $F\subset M$ such that $r(A)>r(F)$ for any $A\supset F$ are called {\em flats}
of $M$ of rank $r(F)$.
Matroid $M$ is  {\em loopless} if $r(A)=0$ implies $A=\emptyset$.

The so-called {\em Bergman fan} of a loopless matroid $M$ is a polyhedral fan $\Sigma_M\subset \R^{|M|-1}$ constructed as follows. Choose $|M|$ integer vectors $e_j\subset\Z^{|M|-1}\subset\R^{|M|-1}$, $j\in M$
such that $\sum\limits_{j\in M} e_j=0$ and any $|M|-1$ of these vectors form a basis of $\Z^{|M|-1}$.
To any flat $F\subset M$ we associate a vector 
$$e_F:=\sum\limits_{j\in F} e_j\in\R^{|M|-1}.
$$
E.g, $e_M=e_\emptyset=0$, but $e_F\neq 0$ for any other (proper) flat $F$.
To any flag of flats $F_{i_1}\subset\dots\subset F_{i_k}$ 
we associate a 
convex cone generated by $e_{F_{i_j}}$. We define $\Sigma_M$ to be the union of such cones, which is, clearly, an $(r(M)-1)$-dimensional integral simplicial fan. It is easy to check (cf. \cite{AK}) that it satisfies
the balancing condition, so that $\Sigma_M$ is a tropical space,
called the {\em Bergman fan} of $M$.

The matroid $M$ is called {\em uniform} if $r(A)=|A|$ for any $A\subset M$. Note that the Bergman
fan of a uniform matroid is a complete unimodular fan in $\R^{|M|-1}$ with $|M|$ maximal cones.

\begin{definition}
A tropical space $X$ is called {\em smooth}, or a {\em tropical manifold},
if all its charts $\phi_\alpha$ 
are open embeddings to
$Y_\alpha = \Sigma _M \times \T^{s}\subset\T^{|M|-1}\times\T^s$
for some loopless matroid $M$ and a number $s\ge 0$. (Here $s$ is the maximal sedentarity in this chart
and $n=r(M)-1+s$ is the dimension of our tropical manifold $X$.)
\end{definition}

Tropical manifolds can be thought of as tropical spaces without points of multiplicity greater than 1,
see \cite{MR}, thus we use the term {\em smooth}. Note that smoothness is a property of the tropical space $(X,\OO_X)$ alone, it does not involve presentation of $X$ as a polyhedral complex.


\section{Homology groups}
 
\subsection{Singular tropical homology}
Let $x\in X$ be a point in a tropical space. 
Choose a sufficiently small open set $U\ni x$ and an embedding $\phi :U \to Y \subset \T^N$. Then for points $y$ such that $\phi(y)$ lies in an adjacent face to $\phi(x)$ we have a natural map between lattices in the tangent spaces $\iota: T_\Z (y) \to T_\Z(x)$, cf. \eqref{eq:strata_maps}.

\begin{definition}\label{def-fk}
The group $\F_k(x)$ is defined as the subgroup of the $k$th exterior power $\Lambda^k(T_\Z(x))$ generated
by the products $\iota(v_1)\wedge\dots\wedge \iota(v_k)$ with
$v_1,\dots,v_k\in T_\Z(y)$ for a point $y$ such that $\phi(y)$ lies in an adjacent face to $\phi(x)$ of the same sedentarity. It is important that all $k$ elements $v_j$ come from a single adjacent face.
The group $\F^k(x)$ is defined as $\Hom (\F_k(x),\Z)$.
\end{definition}

The discussion at the end of Section \ref{sec:stratification} tells us that the groups $\F_k(x)$ and $\F_k(y)$ are canonically identified if $x$ and $y$ belong a single chart $U_\alpha$ and lie in a single stratum $\s$ of $X$. 
Furthermore, if for two points $x,y$, still in the same chart, we have $\s_x\succ\s_y$, then there are
natural homomorphisms
\begin{equation}\label{cosheafmap}
\iota:\F_k(x)\to\F_k(y).
\end{equation}

If three points $x,y,z\in U$ lie in the strata with incidence $\s_x\succ \s_y \succ \s_z$ then the three corresponding maps \eqref{cosheafmap} form a commutative diagram. In other words, if we consider the set of strata in the $U_\alpha \subset X$ as a category (under inclusions) then $\F_k$ forms a contravariant functor from strata of $U_\alpha$ to abelian groups (cf. Proposition \ref{functor-sheaf}).

We may interpret our data as a system of coefficients suitable to define singular homology groups on $X$. Namely, we consider the finite formal sums 
$$\sum \beta_\sigma \sigma,
$$ 
where each $\sigma:\Delta\to X$ is a singular $q$-simplex which has image in a single chart $U_\sigma$ and is such that for each relatively open face $\Delta'$
of $\Delta$ the image $\sigma(\Delta')$
is contained in a single combinatorial stratum $\s_{\Delta'}$ of $X$.
Slightly abusing the notations we'll identify the source and the image of $\sigma$ with
the singular simplex $\sigma$ itself and say that $\tau=\sigma|_{\Delta'}$ is a face of $\sigma$.
Here $\beta_\sigma\in \F_k(\s_\Delta \cap U_\sigma)$. 

These chains form a complex $C_\bullet(X; \F_k)$ with the differential $\dd$ given by the standard singular differential followed by the maps \ref{cosheafmap}. We call such compatible singular chains with coefficients in $\F_k$ {\em tropical chains}. 
The groups 
$$H_{p,q}(X)=H_q(C_\bullet(X; \F_p), \dd)
$$
are called the {\em tropical homology} 
groups.

These homology groups is a version of singular homology groups of a topological space $X$
(after imposing the condition of compatibility of singular chains with the charts and combinatorial strata). 

A priori the groups $H_{p,q}(X)$ depend on the covering. Indeed, if we refine the covering the tropical chains will be more restrictive. However the usual chain homotopy arguments apply and show that the resulting homology groups are canonically isomorphic. Thus we can conclude that the tropical homology groups are independent of the covering $\{U_\alpha\}$.

In case $X$ has a polyhedral structure one can require the singular chains to be compatible with the polyhedral face structure on $X$, rather than with its combinatorial structure. Clearly, the homology groups defines by the two complexes are canonically isomorphic. For polyhedral $X$ there are other equivalent ways for constructing tropical homology groups: simplicial, cellular. This is what we are going to consider next.

\subsection{Cellular and simplicial tropical homology} 
We assume $X$ is polyhedral and compact throughout this subsection.
The main advantage of dealing with cellular and simplicial chain groups is that they are finitely generated. This will give an effective way to calculate the tropical homology.

Recall that $X$ comes with a subdivision into convex polyhedral domains. 
We define the cellular  chain  complex
$$
C^{cell}_q(X; \F_p)=\oplus \F_p(\Delta)=\oplus H_q(\Delta,\dd\Delta;\F_p(\Delta)).
$$ 
Here the direct sum is taken over all $q$-dimensional faces $\Delta$ of the subdivision.
The homology $H_q(\Delta,\dd\Delta;\F_p(\Delta))$ of the pair with constant coefficients equals $\F_p(\Delta)$ since each $q$-dimensional face $\Delta$ in $X$ 
is topologically a closed $q$-disk (recall that $X$ is compact).
 
Our next step is to define the boundary homomorphism
$\dd: C^{cell}_q(X; \F_p)\to C^{cell}_{q-1}(X; \F_p)$. The $\dd$ is 
the composition of the maps 
\begin{equation}\label{homcelld1}
H_q(\Delta,\dd\Delta; \F_p(\Delta))\to H_{q-1}(\dd\Delta; \F_p(\Delta))\to
H_{q-1}(\dd\Delta,\dd\Delta\cap\Sk_{q-2}(X);\F_p(\Delta)),
\end{equation}
the isomorphism
\begin{equation}\label{isocelld}
H_{q-1}(\dd\Delta,\dd\Delta\cap\Sk_{q-2}(X);\F_p(\Delta))
\to
\oplus
H_{q-1}(\Delta',\dd\Delta';\F_p(\Delta)),
\end{equation}
where the direct sum is taken over all $(q-1)$-dimensional subfaces $\Delta'\prec\Delta$, 
and
\begin{equation}\label{homcelld2}
\oplus H_{q-1}(\Delta',\dd\Delta';\F_p(\Delta)) \to
\oplus H_{q-1}(\Delta',\dd\Delta';\F_p(\Delta')).
\end{equation}
In \eqref{homcelld1} the first homomorphism is the boundary homomorphism of the pair $(\Delta,\dd\Delta)$ and
the second one is induced by the inclusion of the pairs $(\Delta,\emptyset)\subset (\Delta,\dd\Delta)$.
The isomorphism \eqref{isocelld} comes from the excision as the quotient space $\dd\Delta/(\dd\Delta\cap\Sk_{q-2}(X))$
is homeomorphic to a bouquet of $(q-1)$-dimensional spheres, one sphere for each $(q-1)$-dimensional
subface $\Delta'\prec\Delta$. 
Finally, the homomorphism \eqref{homcelld2} is induced by \eqref{cosheafmap}.

The homology groups of the cellular chain complex $(C^{cell}_\bullet(X; \F_p), \partial)$ are called the cellular tropical homology groups $H^{cell}_\bullet(X; \F_p)$.
If one has $X$ covered by the open stars of vertices we have the following identification. 

\begin{proposition}
The cellular tropical homology groups $H^{cell}_\bullet(X; \F_p)$ are canonically isomorphic
to the (singular) tropical homology groups $H_\bullet(X; \F_p)$.
\end{proposition}
\begin{proof}
As in algebraic topology with constant coefficients to prove this isomorphism we need to use cellular homotopy.
Let us recall that by the cellular homotopy argument the inclusion
$\Sk_q(X)\to X$ induces an epimorphism
\begin{equation}\label{skq}
H_j(\Sk_q(X);\F_p)\to H_j(X;\F_p)
\end{equation}
for $j\le q$ (which is an isomorphism for $j< q$).
Note that even though $\F_p$ is not a constant coefficient system, all cellular
homotopy takes place within a single cell, so the classical argument also holds here.

Consider the homomorphism (in singular homology groups)
induced by the inclusion of pairs $(\Sk_q(X),\emptyset)\subset (\Sk(X),\Sk_{q-1}(X))$
$$H_q (X; \F_p)\to H_q (\Sk_q(X),\Sk_{q-1}(X);\F_p)=C^{cell}_q(X; \F_p).$$
Its image consists of cycles by the construction of the boundary map in the short exact sequence
of the pair and thus it gives us a homomorphism
\begin{equation}\label{skq-cell}
H_q (\Sk_q(X); \F_p)\to H^{cell}_q (X; \F_p).
\end{equation}

Note that by cellular homotopy the kernel of \eqref{skq-cell} coincides with 
the kernel of \eqref{skq} for  $j=q$. To see surjectivity of \eqref{skq-cell} we consider an element $c\in H^{cell}_q(X; \F_p)$.
Subdividing the faces of $X$ into simplices if needed we may represent $c$ by a singular chain in $C_\bullet(\Sk_q(X);\F_p)$,
whose boundary $\dd c$ is null-homologous in 
$$C^{cell}_{q-1}(\Sk_{q-1}(X), \Sk_{q-2}(X);\F_p).
$$
But $H_{q-1}(\Sk_{q-2}(X); \F_p)=0$ by the dimensional reason and thus $\dd c$ must also vanish in $H^{cell}_{q-1}(\Sk_{q-1}(X);  \F_p)$. Thus we may correct $c$ (by adding to it a singular chain in $\Sk_{q-1}(X)$ whose
boundary coincides with $\dd c$) to make it a cycle in $C_\bullet(X;\F_p)$.
\end{proof}

Next observation will be very useful when we define the cap product action by the wave class.

\begin{lemma}\label{lem:divisible}
Let $\gamma= \sum \beta_\Delta \Delta$ be a cellular cycle in a compact tropical polyhedral space $X$. Then each $\beta_\Delta$ is divisible by the divisorial vectors of $\Delta$.
\end{lemma}
\begin{proof}
We only have to check this for infinite cells $\Delta$. Since $X$ is compact, $\Delta$ must have a boundary face $\Delta_q$ (of sedentarity one higher) for every divisorial direction $q$. But the coefficient  of $\partial \gamma$ at  $\Delta_q$ comes only from the projection of $\beta_\Delta$ along $q$. 
\end{proof}

There is a simplicial variant of the tropical homology arising from the first baricentric simplicial chains on $X$. (The baricentric subdivision of $X$ was described at the end of Section \ref{sec:polyhedral}). Then we can consider the baricentric simplicial chain complex with coefficients in $\F_p$ as a subcomplex $C^{bar}_\bullet(X; \F_p)$ of $C_\bullet(X; \F_p)$. 

Note that the cellular chain complex $C^{cell}_\bullet(X; \F_p)$
can be viewed as a subcomplex of $C^{bar}_\bullet(X; \F_p)$, where all coefficients on simplices of the same cell are taken equal. 
Applying the standard chain homotopy arguments for constant coefficients one can show that this inclusion 
$$C^{cell}_\bullet(X; \F_p) \hookrightarrow C^{bar}_\bullet(X; \F_p)
$$ 
is again a quasi-isomorphism. This allows us to identify both baricentric simplicial and cellular homology with the tropical homology.

\begin{remark}
In \cite{IKMZ} it is shown that in the case when $X$ is a smooth projective tropical manifold that comes
as the limit of a complex 1-parametric family the groups $H_{p,q}(X)$ can be
obtained from the limiting mixed Hodge structure of the approximating family. In particular, we have the equality
$$h^{p,q}(X_t)=\operatorname{rk}H_{p,q}(X),$$
for the Hodge numbers $h^{p,q}(X_t)$ of a generic fiber $X_t$ from the approximating family.
\end{remark}

\begin{remark}
In Section \ref{section:konstruktor} we will show that there is a fairly small subcomplex of $C^{bar}_\bullet(X; \F_p)$, called konstruktor, which suffices to calculate the homology groups $H_{p,q}(X)$ in the smooth projective realizable case.
\end{remark}

\subsection{Tropical cohomology groups}
Finally we define tropical {\em cochains} $C^\bullet(X; \F^p)$ to be certain linear functionals on charts/strata compatible $\Z$-singular chains with values in $\bigoplus_{\alpha, \s} \F^p(\s\cap U_\alpha)$. Namely, if a simplex $\sigma$ lies in $\s \cap U_\alpha$ then we require the value of the cochain to lie in $\F^p(\s \cap U_\alpha)$. If $\sigma$ also lies in $U_\beta$, then its value in $\F^p(\s \cap U_\beta)$ should coincide with its value in  $\F^p(\s \cap U_\alpha)$ via the differential of the overlapping map.

Then one can define the differential as the usual coboundary followed by the maps dual to \eqref{cosheafmap} 
$$
\delta \alpha (\sigma) = \alpha (\dd \sigma) \in \bigoplus_{\tau\subset \sigma} \F^p(\Delta_\tau) \rightarrow \F^p(\Delta_\sigma).
$$
We can define the {\em tropical cohomology} groups 
$$H^{p,q}(X)=H^q(C^\bullet(X; \F^p), \delta).
$$

\subsection{Sheaf/cosheaf (co)homology}
To make connections with sheaf (co)homology theories we use the coefficient systems $\F_p$ to define a constructible cosheaf with respect to the combinatorial stratification of  $X$. With a slight abuse of notations we denote this cosheaf also by $\F_p$.  A cosheaf is a suitable notion to take homology, just like sheaf  for cohomology.

First we construct the pre-cosheaf in each open chart $U_\alpha$.
Given an open set $U\subset U_\alpha$ we consider the poset formed by the connected
components of intersections of the strata of $U_\alpha$ with $U$. The order is given by adjacency.
This poset can be represented by a quiver (oriented graph) $\Gamma(U)$.
Each vertex $v\in\Gamma(U)$ corresponds to a connected component of the intersection $U\cap \s$
of the open set $U$ and a stratum $\s$ of $U_\alpha$. A single stratum can produce several vertices in $\Gamma(U)$, see Figure \ref{cosheaf-explanation}.

To each vertex $v$ we associate the coefficient group $\F_p(v)=\F_p(\s)$.
To an arrow from $v$ to $w$ we associate 
the relevant homomorphism $i_{vw}: \F_p(v)\to \F_p(w)$ from \eqref{cosheafmap}. The groups $\F_p(v)$ with maps $i_{vw}$ thus form a representation of the quiver $\Gamma(U)$.

\begin{figure}
\centering
\includegraphics[height=50mm]{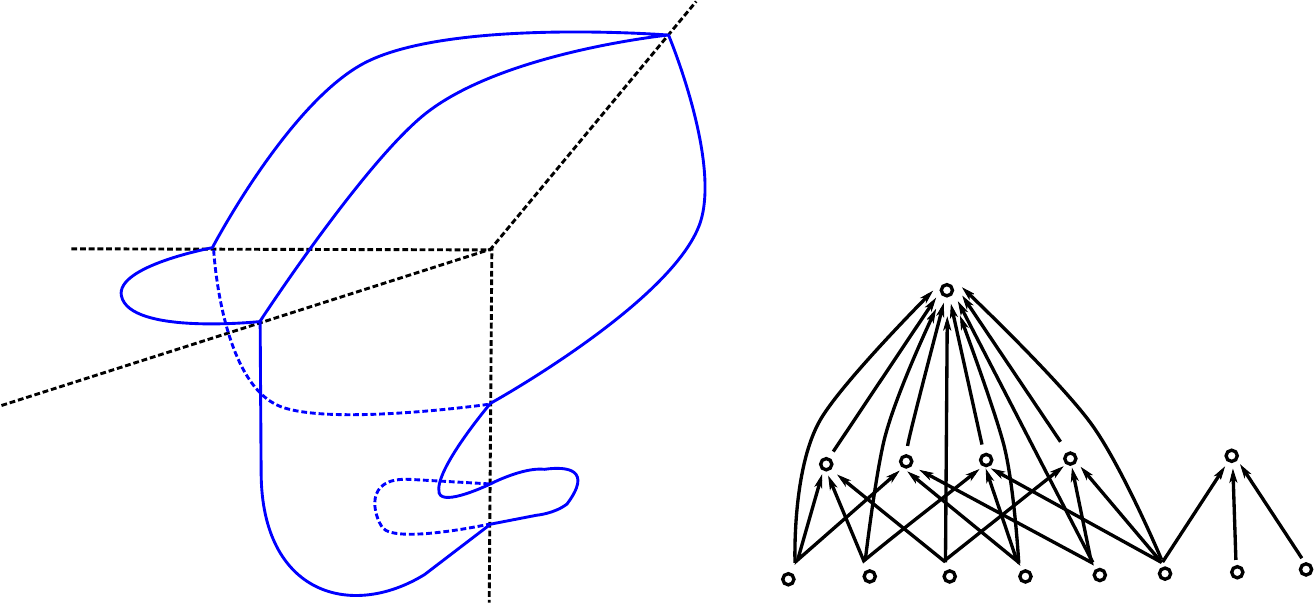}
\caption{\label{cosheaf-explanation} An open set in a polyhedral complex and the corresponding quiver. Here $\F_1(U) \cong \Z^4$.}
\end{figure}

\begin{definition}\label{def-fu}
$\F_p(U)$ is the quotient of the direct sum $\bigoplus_{v\in\Gamma(U)} \F_p(v)$ by the subgroup generated by the elements $a - i_{vw}(a)$ for all pairs of connected vertices $(v,w)$, and all $a\in \F_p(v)$.
\end{definition}

Note that an inclusion $U\subset V\subset U_\alpha$ induces a morphism between the corresponding quivers $\Gamma(U)\to\Gamma(V)$ with isomorphisms at the corresponding vertices. This map clearly preserves the equivalence relation, and hence descends to the map $\F_p(U) \to \F_p(V)$.
Thus, we get a covariant functor from the open sets $U\subset U_\alpha$ (with morphism given by inclusions) to free abelian groups $U\mapsto \F_p(U)$. It is easy to check that all sequences 
\begin{equation}\label{exact:cosheaf} 
\bigoplus_{i,j} \F_p(U_i \cap U_j) \to \bigoplus_i \F_p(U_i) \to \F_p(U) \to 0,
\end{equation}
where $U=\bigcup U_i$, are exact.
Thus the functor $U\mapsto\F_p(U)$ is a cosheaf (cf., e.g. \cite{Bredon}) on the open set $U_\alpha$.

To define the sheaf $\F^p$ we need a contravariant functor $U\mapsto\F^p(U)$. Let  $\Gamma(U)$ to be the directed graph as before with all arrows reversed. We set
 $\F^p(U)$ to be the subgroups of $\bigoplus_{v\in\Gamma(U)} \F^p(v)$, where the collections of elements $\{a_v\in \F^p(v)\}$ are compatible with all the morphisms dual to  \eqref{cosheafmap}.  Note that these collections are precisely the ones annihilated by the elements $a- i_{vw}(a)$ from the Definition \ref{def-fu}, and thus $F^p(U)=\Hom (\F_p(U), \Z)$. Dualizing the exact sequences \eqref{exact:cosheaf} we see that  the functor $U\mapsto\F^p(U)$ is a (constructible) sheaf on $U_\alpha$.

Finally we can glue together the sheaves and cosheaves defined on all open charts $U_\alpha$ (see, e.g., \cite{hartshorne}, Ch. II, Exer. 1.22, for the sheaf version). We get a well defined cosheaf $\F_k$ and sheaf $\F^k$ on $X$ as long as we have the isomorphisms $\psi_{\alpha\beta}$ between the charts  which satisfy $\psi_{\alpha\beta}\circ \psi_{\beta\gamma} = \psi_{\alpha\gamma}$ (see Proposition \ref{prop:differentials}).

The combinatorial strata of $X$ form a category. Its objects are the strata themselves. There is a unique morphism from $\s$ to $\s'$ if $\s\prec \s$, and no morphisms otherwise. Our reasoning above can be formalized into the following general statement.

\begin{proposition}\label{functor-sheaf}
Suppose $X$ has an open covering $\{U_\alpha\}$ and covariant functors $\F_\alpha$ for each $\U_\alpha$ from the combinatorial strata of $\U_\alpha$ to abelian groups which are compatible on the overlaps in the sense of Proposition \ref{prop:differentials}. Then gluing gives rise to a constructible sheaf on $X$. Contravariant functors  yields a constructible cosheaf on $X$.

If the functors $\F_\alpha$ behave naturally with respect to refinements of the covering $\{U_\alpha\}$ the resulting (co)sheaf $\F$ does not depend on the covering.
\end{proposition}

Finally we can use the sheaf-theoretic or \v{C}ech homology and cohomology for cosheaves $\F_p$ and sheaves $\F^p$. The standard algebraic topology techniques identify all these homology theories with the tropical (co)homology. 

\begin{proposition}
There are natural isomorphisms 
$$H_{p,q}\cong H_q(X, \F_p) \quad \text{and} \quad H^{p,q}\cong H^q(X, \F^p),
$$
 where on the right hand side are the sheaf-theoretic (co)homology groups.
\end{proposition}

\section{Tropical waves}

\subsection{Waves and cowaves}
There is also another collection of sheaves and cosheaves that can be associated to a tropical space $X$.  Recall that for every point $x\in X$ we defined the wave tangent spaces $W(x)$ in Section \ref{sec:tangent}.
\begin{definition}\label{def-wk}
We define $W_k(x)$ as the exterior power $\Lambda^k W(x)$.
We also consider the dual vector space $W^k(x)$.
\end{definition}

In any given chart $U_\alpha$ the $W(x)$ can be canonically identified for points in a single stratum $\s$. Thus we may write $W_k(\s\cap U_\alpha)=W_k(x)$ for any point $x\in \s\cap U_\alpha$.
For a pair $\s\prec\s'$ of two adjacent strata the map \eqref{eq:strata_maps} induces the natural homomorphisms 
\begin{equation}\label{Warrow}
\pi: W_k(\s)\to W_k(\s') \ \ \ \text{and}\ \ \ \hat\pi: W^k(\s')\to W^k(\s).
\end{equation}
By Proposition \ref{functor-sheaf} the coefficient system $W_k$ defines a constructible sheaf $\W_k$ on $X$, whereas the $W^k$ defines a cosheaf $\W^k$ for every integer $k\ge 0$.

\begin{definition}
Tropical {\em wave} and {\em cowave} groups, respectively, are 
\begin{equation}\label{tropwave-defn}
H^q(X; \W_k)
\ \ \ \text{and}\ \ \
H_q(X; \W^k). 
\end{equation}
\end{definition}

Again, we can think of these groups from the sheaf-theoretic point of view or as stratum-compatible singular (co)homology with coefficients in the systems $W_k$ and $W^k$.

\begin{example}\label{eg:nodal}
Let us consider a tropical genus 2 curve $C$ with a simple double point.
The underlying topological space of $C$ is a wedge of two circle, i.e. it is a graph with a single vertex $v$
and two edges that are glued to $v$, see Figure \ref{fig:genus2singular}.

The tropical structure in the interior of each edge is isomorphic to an open interval of finite length in $\R$
(treated as the tropical torus $\T^\times=\T\setminus\{-\infty\}$).
 \begin{figure}
    \centering
    \includegraphics[width=2in]{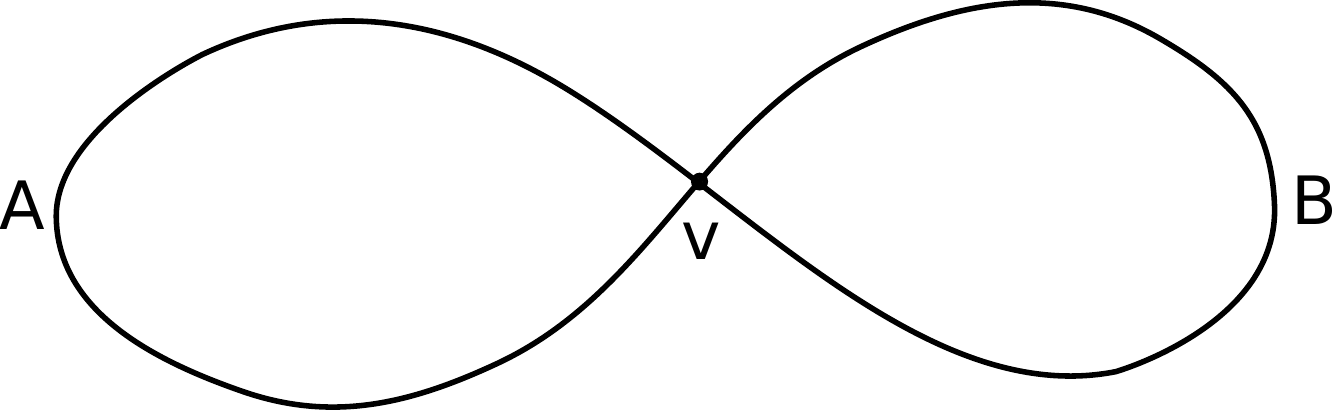}
       \caption{Nodal genus 2 curve.}
      \label{fig:genus2singular}
  \end{figure}
The tropical structure at the vertex $v$ is such that the four primitive vectors divide into 2 pairs of opposite vectors. This means that the chart at $v$ is given by a map to $\R^2$ such that a
neighborhood of $v$ in $C$ goes to the union of coordinate axes and
the four primitive vectors near $v$ go to the unit tangent vectors to those axes.

Thus, $\F_1(v)=\Z^2$ and $W_1(v)=0$. On the other hand every point $x$ in the interior of either edge has the groups $\F_1(x)=\Z$ and $W_1(x)=\R$. The group $\F_0(x)$ is always $\Z$ and $W_0(x)=\R$ for any point $x$.
From the two term cell complex one can easily calculate 
$$
H_0(C; \F_0)\cong \Z, \quad H_0(C; \F_1)\cong \Z, \quad H_1(C; \F_0)\cong \Z^2 ,\quad H_1(C;\F_1)\cong \Z,
$$
and 
$$
H^0(C; \W_0)\cong \R, \quad H^0 (C; \W_1)=0, \quad H^1(C; \W_0)\cong \R^2, \quad H^1 (C; \W_1) \cong \R^2.
$$ 
\end{example}

In general,  $\F_0$ and $\W_0$ are constants, thus for $p=0$ we recover the ordinary topological homology and cohomology groups.

\begin{proposition}
We have $H_{0,q} =H_q(X;\Z)$, $H_{q}(X;\W^0)= H_q(X;\R)$, $H^{0,q}=H^q(X;\Z)$, $ H^{q}(X;\W_0)= H^q(X;\R)$. 
\end{proposition}

\subsection{Pairing of $\F$ and $\W$}
The importance of the wave classes stems from their action on the tropical homology via a natural bilinear map
$$\cap: H^r (X; \W_k) \otimes H_{q}(X; \F_p \otimes \R)  \to H_{q-r}(X; \F_{p+k} \otimes \R)$$
which we are going to define now.
On the chain level this map is just the standard cap product between singular chains and cochains coupled with the wedge product on the coefficients $\wedge: W_k \otimes \F_p \to \F_{p+k} \otimes \R$.

Let us clarify the meaning of the wedge multiplication. For a mobile point $x\in X$ the wave tangent space $W(x)$ is naturally a subspace in $T(x)$ hence the product makes sense on the nose. For any sedentary point $x$ the wave space $W(x)$ naturally projects (along the divisorial directions) to $W'(x)$, which is a subspace of $T(x)$. When taking the wedge product we first apply this projection.

In details, let $\alpha$ be a compatible  $r$-cochain with coefficients in $W_k$ and
$\gamma=\sum \beta\sigma$
be a tropical $q$-chain with coefficients in $\F_p$. For each singular simplex $\sigma$ we denote by $\sigma_{0\dots r}$ its first $r$-face (spanned by the first $r+1$ vertices of $\sigma$) and by $\sigma_{r\dots q}$ its last $(q-r)$-face. Then we set
\begin{equation}\label{eq:cap_product}
\alpha\cap \gamma = \sum  (\alpha(\sigma_{0\dots r})\wedge \beta) \sigma_{r\dots q} .
\end{equation}
Here we push the value of $\alpha$ at the face $\sigma_{0\dots r}$ to the simplex $\sigma$ with the sheaf map and
then push the value of the result from $\sigma$ to $\sigma_{r\dots q}$ using the cosheaf map.

We will need the following local observation. Let us assume we live in a single chart $U_\alpha$.

\begin{lemma}\label{lemma:wedge} 
Let $\s'\prec \s$ be a pair of adjacent strata in $U_\alpha$. Then the diagram 
$$\xymatrix{
W_k(\s)  \otimes \F_p(\s) \ar@<2pc>[d]^{\iota} \ar[r]^--{\wedge} & \F_{p+k}(\s)  \otimes \R\ar[d]^{\iota}\\
W_k(\s') \ar@<2pc>[u]^{\pi} \otimes \F_p(\s')  \ar[r]^--{\wedge} & \F_{p+k}(\s')  \otimes \R
}$$
is commutative in the sense that for any $\alpha\in W_k(\s')$ and $\beta \in \F_p(\s)$ one has $\iota( \pi (\alpha) \wedge \beta)= \alpha \wedge \iota (\beta)$. 
\end{lemma}
\begin{proof}
The wedge product is bilinear with respect to inclusion and quotient (in fact, all) homomorphisms  between free abelian groups.
\end{proof}

\begin{proposition}
For each $r\le q$ the cap product \eqref{eq:cap_product} descends to a natural bilinear map in homology
$$\cap: H^r (X; \W_k) \otimes H_{q}(X; \F_p \otimes \R)  \to H_{q-r}(X; \F_{p+k} \otimes \R).$$
\end{proposition}
\begin{proof}
The statement follows at once from the usual Leibnitz formula 
$$(-1)^r \dd(\alpha \cap \gamma)=   (\delta \alpha) \cap \gamma + \alpha \cap \dd \gamma.
$$
Note that the wedge products in $\delta(\alpha \cap \gamma)$ and $(\partial \alpha) \cap \gamma$ are taken in $\sigma$ and then pushed to $\F_{p+k} (\sigma_{r\dots \hat i \dots q})$. On the other hand the wedge products in $\alpha \cap \delta \gamma$ are taken in $\sigma_{0\dots \hat i \dots q}$ and then pushed to $\F_{p+k} (\sigma_{r\dots \hat i \dots q})$. But Lemma \ref{lemma:wedge} allows us to identify the results.
\end{proof}

\subsection{The group $H^1(X;\W_1\otimes\R)$ and deformations of the tropical structure of $X$}
In this section we assume that $X$ is compact. Recall that $X$ has a covering by charts $\phi_\alpha: U_\alpha \to Y_\alpha \subset \T^{N_\alpha}$. The transition maps on the overlaps are given by integral affine maps $\psi_{\alpha\beta}: \T^{N_\alpha} \dashrightarrow \T^{N_\beta}$.

As a topological space $X$ can be presented as the quotient of the disjoint union of its covering sets $\bigsqcup_\alpha (U_\alpha)$ by the following equivalence relation. We say two points $x\in U_\alpha$ and $y\in U_\beta$ are equivalent if $\psi_{\alpha\beta} \circ \phi_\alpha(x)=\phi_\beta(y)$. Reflexivity of equivalence says that $\psi_{\alpha\beta} = \psi_{\beta\alpha}^{-1}$ (as partially defined maps). Transitivity translates as the cocycle condition, or as the composition rule, $\psi_{\beta\gamma}\circ \psi_{\alpha\beta} = \psi_{\alpha\gamma}$.

Conversely, given open subsets $\phi_\alpha(U_\alpha)\subset Y_\alpha \subset \T^{N_\alpha}$ and a collection of integral affine maps $\psi_{\alpha\beta}$ satisfying $\psi_{\alpha\beta} = \psi_{\beta\alpha}^{-1}$ and $\psi_{\beta\gamma}\circ \psi_{\alpha\beta} = \psi_{\alpha\gamma}$ we can define a topological space $X$ as the quotient of $\bigsqcup_\alpha \phi_\alpha(U_\alpha)$ by the equivalence given by the $\psi$'s. 

$X$ will be a tropical space provided all subsets $\phi_\alpha(U_\alpha)$ remain open in the quotient and $X$ satisfies the finite type condition. Moreover we will get an isomorphic tropical space if the $\psi_{\alpha\beta}$ are changed by a "coboundary" (twisted by automorphisms $\psi_\alpha: \T^{N_\alpha} \dashrightarrow \T^{N_\alpha}$ for some $\alpha$).

Let $\tau$ be a class in $H^1(X;\W_1)$. We can assume that the covering $\{U_\alpha\}$ is fine enough so that $\tau$ can be represented by a \v Cech 1-cocycle $\tau_{\alpha\beta}\in W (U_\alpha \cap U_\beta)$. We can also assume that all $U_\alpha$ and $U_\alpha\cap U_\beta$ are connected. Then $W (U_\alpha \cap U_\beta)$ consists of vectors parallel to all mobile strata in $U_\alpha \cap U_\beta$. We can think of $W (U_\alpha \cap U_\beta)$ as a subspace in $\R^{N_\alpha} \subset \T^{N_\alpha} $ via the map $\phi_\alpha$, or in $\R^{N_\beta} \subset \T^{N_\beta}$ via $\phi_\beta$.

By shrinking the $U_\alpha$ if necessary it will also be convenient to assume that slightly larger open subsets $V_\alpha\supset \overline{U_\alpha}$ do not contain any new strata other than those already in the $U_\alpha$. 
For instance if $X$ is polyhedral we can take $U_\alpha$ to be the open stars of vertices, "shrunk" a little bit.

Now for $\epsilon >0$ we modify the overlapping maps $\psi_{\alpha\beta}: \T^{N_\alpha} \dashrightarrow \T^{N_\beta}$ by precomposing them with the translation by $\epsilon\tau_{\alpha\beta}$. Since $\tau_{\alpha\beta}=-\tau_{\beta\alpha}$ the new relation is reflexive. Also since $\tau$ is a cocycle the new maps $\psi^{\epsilon\tau}_{\alpha\beta}$ satisfy the composition rule. Thus they define a new equivalence relation and we call the corresponding quotient space $X_{\epsilon\tau}$ the deformation of $X$.

\begin{proposition}
For $\epsilon >0$, small enough,  $X_{\epsilon\tau}$ is a tropical space.
\end{proposition} 
\begin{proof}
We only need to show that $X_{\epsilon\tau}$ is of finite type and each of the $U_\alpha$ is still an open subset in $X_{\epsilon\tau}$. For the latter it is enough to show that each $U_\alpha\cap U_\beta$ is open. But this is clear since by condition that slight enlargements of  $U_\alpha$ contain no new strata, no new strata can appear in $U_\alpha\cap U_\beta$ for small enough $\epsilon$. The argument for the finite type condition is similar.  
\end{proof}

The deformed tropical space is especially easy to visualize in the polyhedral case. Namely, $X_{\epsilon\tau}$
has the same combinatorial face structure, but the faces of $X_{\epsilon\alpha}$ themselves
may have different shapes and sizes. E.g. if $X$ is 1-dimensional, the lengths of the edges of $X$ and $X_{\epsilon\alpha}$ may be different.


\section{Straight classes}
\subsection{Straight cycles in tropical homology} 
We start with a natural generalization of balanced polyhedral complexes in $\T^N$
to a situation where a facet can have a weight. 
A weighted balanced polyhedral complex $Y\subset\T^N$ is a union of a finite number of facets $D$
as before, but now each $D$ is enhanced with an integer weight $w(D)\in\Z$
subject to the following weighted balancing condition
for every $(n-1)$-dimensional mobile face  $E\subset Y$.
As in \eqref{bal-cond}  we consider all facets $D_1,\dots,D_l\subset \T^{N}$ adjacent to $E$
and take the quotient of $\R^{N}$ by the linear subspace parallel to $E^\circ$,
the non-infinite part of $E$. The weighted balanced condition is
\begin{equation}\label{wbal-cond}
\sum\limits_{k=1}^l w(D_k)\epsilon_k=0.
\end{equation}
We say that the weighted balanced complex $Y$ is 
{\em effective} if the weights of all its facets are positive.

Just as balanced polyhedral complexes form local models for tropical spaces,
effective weighted balanced polyhedral complexes form models for
{\em weighted tropical spaces}.

\begin{definition}
A weighted tropical space is a topological space $X$ enhanced with a weight function
$w:X\dashrightarrow\N$ defined on an open dense set $A\subset X$ and a sheaf $\OO_X$
of functions to $\T$ such that there exists a finite covering of compatible charts
$\phi_\alpha: U_\alpha\to Y_\alpha \subset  \T^{N_\alpha}$ with the following properties.
\begin{itemize}
\item $Y_\alpha$ is an effective weighted balanced polyhedral complex in $\T^{N_\alpha}$.
\item For the relative interior $D^\circ$ of any facet $D\subset Y_\alpha$ we
have $\phi_\alpha^{-1} (D^\circ)\subset A$ while the weight function $w$ is constant
on  $\phi_\alpha^{-1}(D^\circ)$ and equal to the weight of $D$.
\item
For each facet $D\subset Y_\alpha$ there exists a $\Z$-linear transformation
$\Phi_D:\Z^{N_\alpha}\to \Z^{N_\alpha}$ of determinant $w(D)$
such that $\OO_X|_{D^\circ\cap U_\alpha}$ is induced by $\Phi_D^{-1}\circ\phi_\alpha$.
\end{itemize}
\end{definition}

We may reformulate the last condition of this definition by saying
that each facet $D$ comes with a sublattice of index $w(D)$ of
the tangent lattice $T_{\Z}(x)$, $x\in D$. This sublattice is locally
constant and does not depend on the choice of charts.
Note that not every weighted balanced polyhedral complex in $\T^N$
is a weighted tropical space in this sense as it is not always possible
to consistently choose such a sublattice.
However no such sublattice for the facets of $Z$ is needed for the following definition.

\begin{definition}[cf. \cite{MR}, \cite{Shaw}]
Let $X$ be a tropical space. 
A subspace $Z\subset X$ enhanced with a weight function
$$w:Z\dashrightarrow\Z$$ defined on an open dense set $A\subset Z$
is called a {\em straight tropical $p$-cycle} if for every chart
$\phi_\alpha:U_\alpha\to Y_\alpha\subset\T^{N_\alpha}$ of $X$
there exists a weighted $p$-dimensional balanced polyhedral complex $Z_\alpha\subset\T^{N_\alpha}$
such that $\phi_\alpha: Z\cap U_\alpha \to Z_\alpha$ is an open embedding, and
for the relative interior $D^\circ$ of any facet $D\subset Z_\alpha$ we
have $\phi_\alpha^{-1} (D^\circ)\subset A$. The weight function $w$ is constant
on  $\phi_\alpha^{-1}(D^\circ)$ and equal to the weight of $D$.
\end{definition}

\begin{proposition}\label{specialfundclass}
Each straight tropical $p$-cycle $Z\subset X$ gives rise to a canonical element
$[Z]\in H_{p,p}(X)$ in the tropical homology group of $X$.
\end{proposition}
\begin{proof}
We choose a sufficiently fine (topological) triangulation of $Z=\bigcup \sigma$
so that each $p$-simplex $\sigma$ from the triangulation lies in a single chart $U_\alpha$
and in a single combinatorial stratum $\s$ of $X$.
In particular each $\sigma$ carries the weight $w(\sigma)$ induced from $Z$.
An orientation of $\sigma$ defines the canonical volume element $\Vol_\sigma\in\F^X_p(\sigma)$
given by the generator of $\Lambda^p(W^Z_{\Z}(\sigma))\cong\Z$.
Inverting the orientation of $\sigma$ will simultaneously invert the sign of $\Vol_\sigma$.
Thus the product $\Vol_\sigma \sigma$ is a well-defined tropical chain in $C_p (X;\F_p)$.
Then the weighted balancing condition for $Z$ ensures that 
$$\gamma_Z= \sum_{\sigma\subset Z} w(\sigma) \Vol_\sigma \sigma
$$
is a cycle in $C_p (X;\F_p)$. Its class is clearly independent of the triangulation and gives the desired element $[Z]\in H_p(X;\F_p)=H_{p,p}(X)$.
%
%
\end{proof}

\begin{definition}
Elements of $H_{p,p}$ realised by straight tropical cycles
as in Proposition \ref{specialfundclass} are
called straight homology classes (or, in other existing terminology, 
{\em special} or {\em algebraic}).
They form a subgroup
$$H_{p,p}^{straight}(X) \subset H_{p,p}(X).$$
\end{definition}

\begin{example}
Recall that the tropical $N$-dimensional projective space $\tp^N$ may be obtained
by gluing $N+1$ affine charts $\T^N$ with the help of integral affine maps, cf. e.g. \cite{MR}.
A topological subspace $X\subset\tp^N$ is called a {\em projective
tropical space} if the intersection of $X$ with any such chart is a balanced polyhedral complex.
A projective tropical space has a non-trivial straight homology class
$$[H^X_p]\in H_{p,p}^{straight}(X)$$
(called the {\em hyperplane section}) in any dimension $p=0,\dots,n=\dim X$.

To see this we start from the case $X=\tp^N$.
Consider the equations $x_{j}=c_j$, $j=p+1,\dots,n$, $c_j\in\R$,
in a chart $\T^N\subset\tp^{N}$. They define a $p$-dimensional
linear space parallel to a coordinate plane.
We may take for $H_p$ the topological closure of this linear space in $\tp^N$.
Clearly, the homology class $[H_p]$ does not depend on the choice of
the $\T^N$-chart or on permutation of coordinates in this chart. Furthermore,
$H_0$ is a point and thus $[H_0]\neq 0$ in $H_{0,0}(\tp^N)\cong \Z$. Note that this also
implies that $[H_p]\neq 0$ in $H_{p,p}(\tp^N)$ as we may choose the transverse
representatives $H_p$ and $H_{p'}$ so that $H_p\cap H_{p'}=H_{p+p'-N}$,
cf. \cite{Shaw}. It is easy to show that any element of
$H_{*,*}(\tp^N)$ is generated by $[H_p]$, $p=0,\dots,N$.

A similar construction can be made for general projective tropical spaces $X\subset \tp^N$.
We take $H^X_p=H_{N+p-n}\cap X$ where $H_{N+p-n}$ is chosen to be transverse to $X$
with the help of translations in $\R^N$. But in addition to those hyperplane sections
and their powers $H_{*,*}(X)$ may have additional, more interesting, straight classes. 
\end{example}

\subsection{Straight cowaves}
A notion of straight classes exists also for cowaves.
Once again, let $Z\subset X$ be a subspace such that 
each chart $\phi_\alpha$ takes $Z$ to a $q$-dimensional polyhedral complex in $\T^{N_\alpha}$
(which we no longer assume balanced). We refer to such subspace of $X$ as a {\em straight subspace}.

In this subsection we assume that 
$\dim W(x)=m$ for some $m$ almost everywhere on $Z$.
In other words we assume that each open facet of $Z$ sits
in the $m$-skeleton of $X$, but outside of the $(m-1)$-skeleton of $X$.
We call such straight subspaces $Z$ 
{\em purely $m$-skeletal}. E.g. $Z$ is $n$-skeletal if no open facets of $Z$ intersect
$\operatorname{Sk}_{n-1}(X)$.

\begin{definition}
A {\em coweight function} on $Z$ is a function
$$x \mapsto cow(x)\in W^m(x)$$
defined on an open dense set $A\subset Z$. 
Here we assume that $\dim W(x)=m$
whenever $x\in A$, so we have $W^m(x)\approx\Z$.
\end{definition}
This is a dual notion to the weight function. But while the weight function was integer-valued,
here we do not have a canonical isomorphism between $W^m(x)$ and $\Z$, 
it is only canonical up to sign.

Let $x\in A$ be inside of a facet of $Z$ parallel to a $q$-dimensional affine space $L$ (in a chart $\phi_\alpha$).
As in the previous subsection, we may
consider the volume element $\Vol_L\in W_q(x)$ which is well-defined by the integer lattice in $L$
and a choice of orientation of $L$.
Given this choice we have a well-defined map
$$\lambda\mapsto cow(x)(\lambda\wedge\Vol_L),$$
$\lambda\in W_{m-q}(x)$ and thus an element in $W^{m-q}(x)$,
a group that depends only on the open facet of $Z$ containing $x$.
Thus any $q$-simplex $\sigma$ embedded to the same facet and parallel to $L$
defines a canonical chain with coefficients in $W^{m-q}(x)$.

In particular, a triangulation of a coweighted purely $m$-skeletal $q$-dimensional polyhedral pseudocomplex $Y$
gives rise to a cowave chain in $C_q(X;\W^{m-q})$.
Such cowave chains are called {\em straight}.

As in Proposition \ref{specialfundclass} we may associate a singular chain
with the coefficients in $W^{m-q}$ to $Z$ by using a combinatorial stratification of $Z$.
\begin{definition}
A coweighted straight subspace $Z\subset X$ is called {\em cobalanced} if 
the resulting chain is a cycle. We may refine this into a local notion by saying
that $Z$ is cobalanced at $x\in Z$ if $x$ is disjoint from the support
of the boundary of the resulting special cowave cochain.
\end{definition}

Note that once an orientation of $W(x)$ is chosen we may identify
coweight and weight at $x$. 

\begin{proposition}
Suppose that a $q$-dimensional coweighted straight subspace $Z$
is purely $m$-skeletal and that $x\in Z$ belongs to a relative interior of
a $(q-1)$-dimensional face (in a chart) with $\dim W(x)=m$.
Then $Z$ is cobalanced at $x$ if and only if $Z$ is balanced at $x$.
\end{proposition}
\begin{proof}
Note that $x$ must belong to the same combinatorial stratum of $X$ as its small
open neighbourhood in $Z$, 
since $Z$ is purely $m$-skeletal and $\dim W(x)=m$. Thus $\W^m|_Z$ is locally trivial near $x$
and we may translate coweights into weights simultaneously for the whole neighbourhood
with the help of an arbitrary orientation of $W(x)$. 
\end{proof}

At the same time if $\dim W(x)<m$ then the cobalancing condition is different from
the balancing condition. We believe that study of straight cowaves might be
useful, particularly in the context of mirror symmetry.

\section{The eigenwave}

\subsection{The eigenwave $\phi$}
There is a canonical element $\phi\in H^1(X;\W_1)$ for every compact tropical space $X$.
Unfortunately, it does not have a preferred representative as a singular wave cocycle
in the case when $X$ has points of positive sedentarity.
Rather we shall represent it in the quotient space $C^1(X;\W_1)/B_{div}^1(X;\W_1)$ where the subspace $B_{div}^1(X;\W_1) \subset C^1(X;\W_1)$ will consist of certain coboundaries. Note that this ambiguity
will not cause us any problems with the cap product of $\phi$ and the homology cycles because, as we will see, taking product with elements in $B_{div}^1(X;\W_1)$ annihilates any singular cycle.

Let $C_{div}^0(X;\W_1) \subset C^0(X;\W_1)$ be the subspace of 0-wave cochains whose values on points $x\in X$ are in $W^{div}(x)$.
We let $B_{div}^1(X;\W_1) \subset C^1(X;\W_1)$ consist of the coboundaries of the cochains from $C_{div}^0(X;\W_1)$.  The elements $\gamma \in B_{div}^1(X;\W_1)$ are characterized by the property that on any singular 1-simplex $\tau$ the values $\gamma(\tau)$ belong to the subspace $W^{div}(\tau) \subset W(\tau)$ spanned by the divisorial subspaces at the boundary points of $\tau$.

We are ready to define the eigenwave class $\phi \in C^1(X;\W_1)/B_{div}^1(X;\W_1)$.
Let us first consider the case when all points of $X$ have zero sedentarity, in particular
$B_{div}^1(X;\W_1)=0$. In such case we define the value of $\phi$ on a singular 1-simplex
$\tau:[0,1]\to X$ as $\tau(1)-\tau(0)$.
Recall that our singular chains are assumed to be compatible with the combinatorial
stratification of $X$ so that $\tau((0,1))$ is contained in a single combinatorial stratum
and a single tropical chart.
This means that the difference $\tau(1)-\tau(0)$ can be interpreted as a vector in the
tangent space to this stratum and therefore in $W(\tau)$.

Returning to the general case, if $x\in X$ is of positive sedentarity
we choose a nearby mobile point $y_x$ which maps to $x$ under the projection
along divisorial directions. If $x\in X$ is mobile we set $y_x=x$.
\begin{definition}
The element $\phi \in C^1(X;\W_1)/B_{div}^1(X;\W_1)$ is defined on a 1-simplex $\tau:[0,1]\to X$
as the vector $w_\tau:= y_{\tau(1)}-y_{\tau(0)} \in W(\tau)$.
\end{definition}
Clearly the ambiguity in $w_\tau$ resulting from different choices of $y_x$ is confined
to $B_{div}^1(X;\W_1)$.
The next proposition asserts that $\phi$ defines a class in $ H^1(X;\W_1)$, which we call the {\em eigenwave} of $X$. We denote this class also by $\phi$, this should not cause any confusion.

\begin{proposition}
$\delta \phi =0$. 
\end{proposition}
\begin{proof}
By definition the value of $\delta\phi$ on a 2-simplex $\sigma$ is the sum of the values of $\phi$ on the three edges $\tau_1,\tau_2,\tau_3$ of $\sigma$. This is clearly zero (perhaps after applying the maps $\pi: W(\tau) \to W(\sigma)$ in case some of the $\tau_i$ land in different strata).
\end{proof}

\subsection{Action of the eigenwave $\phi$ and its powers on tropical homology}
\label{subsection:wave_action}

The $k$-th cup powers of $\phi$ are also (higher degree) wave classes $\phi^k\in H^k(X;\W_k)$.
One can define the value of $\phi^k$ on a $k$-simplex $\sigma$ modulo the ideal in $W_k(\sigma)$ generated by the $W^{div}$ for all vertices in $\sigma$. Namely, for an edge $\tau\prec \sigma$ let $w_\tau \in W(\sigma)$ stand for the vector $y_{\tau(1)}-y_{\tau(0)}\in W(\tau)$ pushed to $W(\sigma)$.
Then
\begin{equation}\label{eq:phi^k}
\phi^k (\sigma)=w_{\sigma_{01}}\wedge\dots \wedge w_{\sigma_{k-1, k}} =: w_\sigma \in W_k (\sigma).
\end{equation}
Taking the cap product with $\phi^k=[\phi^k_{sing}]$ gives us the homomorphism:
\begin{equation}\label{eq:wave_action}
 \phi^k \cap:  H_q(X;\F_p \otimes \R) \to H_{q-k}(X;\F_{p+k} \otimes \R).
\end{equation}

In case $X$ is compact and polyhedral we consider its baricentric subdivision and think of the $H_q(X; \F_p)$ as simplicial or cellular homology groups. The advantage is that we can define the cap product with $\phi^k$ on the cycle level
\begin{equation}\label{eq:wave_action_cycle}
\phi^k: C^{cell}_q(\F_p) \to C^{bar}_{q-k}(\F_{p+k}\otimes \R).
\end{equation}
Below we give two different descriptions of the map \eqref{eq:wave_action_cycle} depending on the choice of vertex ordering. The first result is  a cycle in $C^{bar}_{q-k}(\F_{p+k}\otimes \R)$ while the second one is still in $C^{cell}_{q-k}(\F_{p+k}\otimes \R) \subset C^{bar}_{q-k}(\F_{p+k}\otimes \R)$. 

We recall the notion of the dual cells in the first baricentric subdivision of a polyhedral complex. Let $\Delta\in X$ be a $q$-cell. For any {\em finite} $j$-dimensional face $\Delta' \prec\Delta$ of the sedentarity $s(\Delta')=s(\Delta)$ its {\em dual cell} $\hat\Delta'_\Delta$ in the baricentric subdivision of $\Delta$ is defined as the union of all $(q-j)$-simplices in $bar(\Delta)$ containing the baricenters of $\Delta$ and $\Delta'$. We can think of $\hat \Delta'_\Delta$ as a simplicial $(q-j)$-chain. The orientations of the pair  $\Delta'$ and  $\hat \Delta'_\Delta$ are taken to agree with the original orientation of $\Delta$. 

Let $\gamma= \sum  \beta_\Delta \Delta$ be a cycle in $C^{cell}_q(X;\F_p)$. Then according to Lemma \ref{lem:divisible} the coefficients $\beta_\Delta$ for all $\Delta \subset X$ have to be divisible by the divisorial directions of $\Delta$. In particular, the wedge product of $\beta_\Delta$ with any element in $\wedge^k (W(\Delta)/W^{div}(\Delta))$ gives a well-defined element in $\F_{p+k}(\Delta)\otimes \R$. We can also think of $\gamma=\sum_{\sigma\in bar (\Delta)} \beta_\Delta \sigma$ as an element in $C^{bar}_q(X;\F_p)$. 

{\bf Description 1:} 
We label the vertices of each $q$-simplex $\sigma$ in $bar(\Delta)$ according to the dimension of the largest cells whose baricenters they represent (recall that several faces of $\Delta$ of different sedentarity may have the same baricenter). 
In this case the cycle $\phi^k \cap \gamma \in C^{bar}_{q-k}(X;\F_{p+k})$ is supported on the dual subdivision inside the $q$-skeleton of $X$.

Precisely, for every $k$-face $\Delta'$ of $\Delta$ let $w_{\Delta'}\in W_k(\Delta)$ denote the volume element  associated to $\Delta'$ as in \eqref{eq:phi^k}. 
Clearly, $w_{\Delta'}$ equals the sum of all $w_{\sigma_{0\dots k}}$ (taken with appropriate signs) for the $k$-simplices $\sigma_{0\dots k}$ forming the baricentric triangulation of $\Delta'$.
Then one can easily calculate from the definition of the cap product: 
\begin{equation}\label{eq:wave_action1}
\phi^k \cap (\sum_{\sigma\in bar (\Delta)} \beta_\Delta \sigma) = \sum_{\Delta' \prec\Delta} (w_{\Delta'}\wedge\beta_\Delta) \hat \Delta'_\Delta ,
\end{equation}
where the sum is taken over all $k$-dimensional faces of $\Delta$.
Note that higher sedentary $k$-faces don't appear in the sum because $\beta_\Delta$ vanishes  when pushed to these higher sedentary faces.

\begin{figure}
\centering
\includegraphics[height=28mm]{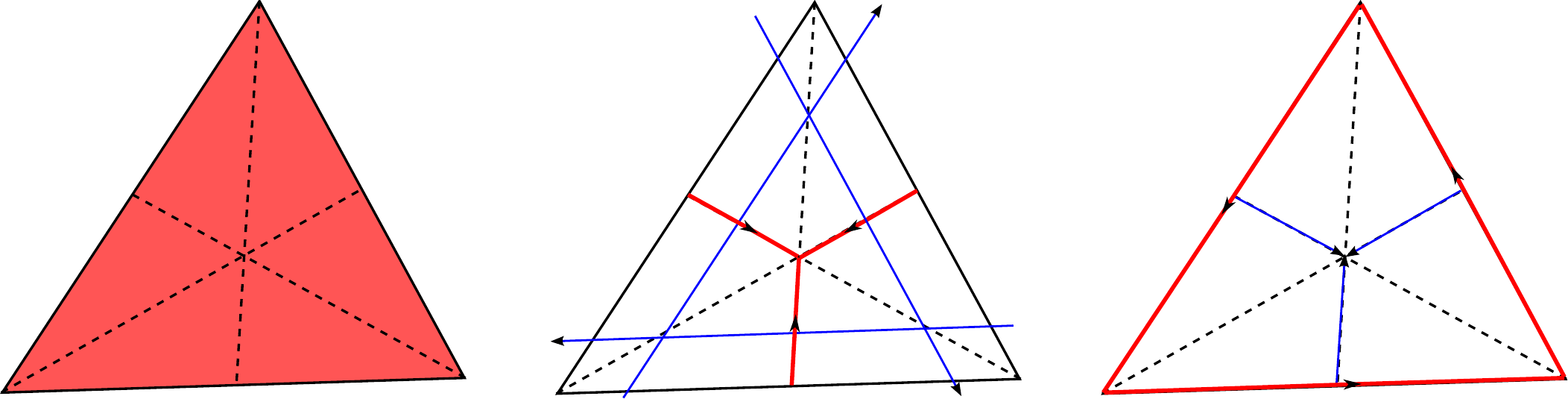}
\caption{The two descriptions of the wave action on a 2-cell $\sigma$. The support of $\phi_{sing} \cap \sigma$ is red and the framing is blue.}\label{fig:wave_action} 
\end{figure}

{\bf Description 2:} 
Here we label the vertices of each $\sigma$ 
in the opposite order to the description 1. That is the baricenters with the smaller numbers correspond to the larger faces. Now
the cycle  $\phi^k \cap \gamma\in C^{bar}_{q-k}(X;\F_{p+k})$ is supported on the $(q-k)$-skeleton of $X$. 

Precisely, for every $(q-k)$-face $\Delta'$ of $\Delta$ let $\hat w_{\Delta'}\in W_k(\Delta)$ denote the polyvector corresponding to the integration along the chain $\hat \Delta'_\Delta$. Note that the faces  $\sigma_{k\dots q}$ lie in the $(q-k)$-faces of $\Delta$. The polyvectors $w_{\sigma_{0\dots k}}$ sum to $\hat w_{\Delta'}$ for those simplices $\sigma\in bar (\Delta)$ whose faces $\sigma_{k\dots q}$ give the same simplex in $bar(\Delta')$. Then again from the definition of the cap product we can write:
\begin{equation}\label{eq:wave_action2}
\phi^k \cap (\sum_{\sigma\in bar(\Delta)} \beta_\Delta \sigma) = \sum_{\Delta' \prec\Delta} (\hat w_{\Delta'}\wedge\beta_\Delta) \Delta',
\end{equation}
where the sum is taken now over all $(q-k)$-dimensional faces of $\Delta$.

It is straight forward to check that in both cases the resulting chain 
$$\phi^k \cap (\sum_\Delta\sum_{\sigma\in bar(\Delta)} \beta_\Delta \sigma)$$ 
is a cycle.

\begin{conjecture}\label{conj:isomorphism}
Let $X$ be a smooth compact tropical variety. Then for $q\ge p$
$$ \phi^{q-p} \cap:  H_q(X;\F_p \otimes \R) \to H_{p}(X;\F_{q} \otimes \R) 
$$
is an isomorphism.
\end{conjecture}

We will prove the conjecture in the realizable case in Section \ref{section:konstruktor} though we believe that realizability assumption is not necessary. Certain amount of smoothness, on the other hand, is essential. In the non-smooth case even the ranks of $H_q(X; \F_p)$ and  $H_{p}(X; \F_{q})$ may not agree. A simple example is provided by the nodal genus 2 curve (see Example \ref{eg:nodal}).

The action of the eigenwave $\phi$ is trivial on straight tropical $(p,p)$-classes.
\begin{theorem}
If $\gamma\in H_{p,p}^{straight}(X)$ then $\phi\cap\gamma=0$.
\end{theorem}
\begin{proof}
Any vector parallel to a simplex $\sigma$ of a special tropical cycle turns to zero after the
wedge product with the volume element of $\sigma$.
\end{proof}


\section{Intermediate Jacobians}

\subsection{Tropical tori}
Let $V$ be a $g$-dimensional real vector space containing two lattices $\Gamma_1,\Gamma_2$ of maximal rank, that is $V \cong \Gamma_{1,2}\otimes \R$. Suppose we are given an isomorphism $Q:\Gamma_1\to \Gamma_2^*$, which is symmetric if thought of as a bilinear form on $V$. 

\begin{definition}
The torus $J=V/\Gamma_1$ is the {\em principally polarized tropical torus} with $Q$ being its polarization. The tropical structure on $J$ is given by the lattice $\Gamma_2$. If, in addition $Q$ is positive definite, we say that $J$ is an {\em abelian variety}.
\end{definition}

\begin{remark}
The map $Q:\Gamma_1\to \Gamma_2^*$ provides an isomorphism of $J=V/\Gamma_1$ with the tropical torus $V^*/\Gamma_2^*$. The tropical structure on the latter is provided by the lattice $\Gamma_1^*$.
\end{remark}

\begin{remark}
The above data $(V, \Gamma_1, \Gamma_2, Q)$ is equivalent to a non-degenerate real-valued quadratic form $Q$ on a free abelian group $\Gamma_1\cong \Z^g$. The other lattice $\Gamma_2 \subset V := \Gamma_1\otimes \R$ is defined as the dual lattice to the image of $\Gamma_1$ under the isomorphism $V\to (V)^*$ given by $Q$.
\end{remark}

Let us take the free abelian group $\Gamma_1=H_q(X; \F_p)\cong\Z^g$ with $p+q=\dim X$, and $p\le q$. We define the tropical intermediate Jacobian as the torus above together with a symmetric  bilinear form $Q$ on $H_q(X; \F_p)$. 

The form $Q$ is a certain intersection product on tropical cycles which we define in two ways. The first definition is manifestly symmetric while the second definition descends to homology. And then we show that the two definitions are equivalent. 

Unfortunately we are not able to show in this paper that the form is non-degenerate, though we believe that in the smooth and compact case this should be true (cf. Conjecture \ref{conj:non_degenerate}).

\subsection{Intersection product} Let $X$ be a compact tropical space of dimension $n$.
For a singular simplex $\sigma$ we denote its relative interior by $int(\sigma)$. We abuse the notation $int(\sigma)$ to denote also its image in $X$. 

\begin{definition}
We say that a tropical chain $\sum\beta_\sigma \sigma\in C_{q}(X; \F_{p})$ is {\em transversal to the combinatorial stratification of $X$} (or, simply, transversal) if for any simplex $\sigma$ and any face $\tau\prec_k\sigma$
we have
\begin{itemize}
\item  $int(\tau)$ meets strata of X only of dimension $(n-k)$ and higher;
\item if $\tau$ lies in a sedentary stratum of $X$ then $\beta_\sigma$ is divisible by all corresponding divisorial directions.
\end{itemize}
\end{definition}

\begin{definition}
We say that two transversal tropical chains $\sum\beta_{\sigma'} \sigma'\in C_{q'}(X; \F_{p'})$ and $\sum\beta_{\sigma''} \sigma'' \in C_{q''}(X; \F_{p''})$ form a {\em transversal pair} if the following holds. For every pair of simplices $\sigma', \sigma''$ from these chains and any choice of their faces $\tau'\prec\sigma', \tau''\prec \sigma''$, if the interiors $int(\tau'), int(\tau'')$ lie in the same stratum $\s$ then $int(\tau'), int(\tau'')$ are transversal in the usual sense as smooth maps to $\s$.
\end{definition}

If a pair of simplices $\sigma', \sigma''$ from the transversal pair have non-empty intersection then all three submanifolds $\sigma' , \sigma'', \sigma'\cap \sigma''$ are supported on the same maximal stratum ${\s_{\sigma'\cap \sigma''}}$ of $X$ (and on no smaller strata). The oriented triple $\sigma' , \sigma'', \sigma'\cap \sigma''$ determines an integral volume element $\Vol_{\s_{\sigma'\cap \sigma''}}$ as well as its dual volume form $\Omega_{\s_{\sigma'\cap \sigma''}}$. By transversality, $\sigma'\cap \sigma''$ has dimension $q'+q''-n$. We can choose a singular chain $\sum \tau$ representing its relative fundamental class agreeing with the orientation of $\sigma'\cap \sigma''$. 

Let $\gamma'=\sum\beta_{\sigma'} \sigma'\in C_{q'}(X; \F_{p'})$ and $\gamma''=\sum\beta_{\sigma''} \sigma'' \in C_{q''}(X; \F_{p''})$ be a transversal pair of tropical chains. We define the following bilinear product with values in the cowave chains:
\begin{equation}\label{eq:dot_product}
\gamma'\cdot\gamma''=\sum_
{\tau\subset\sigma'\cap \sigma''}
 \Omega_{\s_\tau} (\beta_{\sigma'}\wedge \beta_{\sigma''}) \cdot \tau \in C_{q'+q''-n}(X; \W^{n-p'-p''}).
\end{equation}

\begin{remark}
Note that $\gamma'\cdot\gamma''$ has no support on infinite simplices $\tau\subset\sigma'\cap \sigma''$ since the divisorial directions in $W^{div}(\tau)$ divide both $\beta_{\sigma'}$ and $\beta_{\sigma''}$.
\end{remark}

\begin{remark}
If $q'+q'' < n$ or $p'+p''>n$ then $\gamma' \cdot \gamma''=0$ for dimensional reasons. In what follows we will tacitly assume this is not the case.
\end{remark}

From now on we assume that $X$ is a compact smooth tropical space.
Our goal will be to show that in this case the above product descends to homology.


First we show that we can deform all cycles to a transverse position. Since the question is local we can work in a chart $\phi_\alpha: U_\alpha \to Y \subset \T^N$.  The next lemma says that we can move a tropical cycle $\gamma$ off a face $E$ of $Y$, if it intersects it in higher than expected dimension, not changing it outside the open star $\ostar(E)$.

\begin{lemma}\label{transversal-E}
Let $\gamma \in C_q(X,\F_p)$ be a (singular) tropical cycle in a tropical $n$-dimensional manifold $X$ and let $E$ be an $l$-face of $Y$ in a chart $\phi_\alpha: U_\alpha \to Y \subset \T^N$. Then there exists a cycle $\gamma'=\sum \beta_\sigma \sigma\in C_q(X;\F_p)$ homologous
to $\gamma$ and such that for any $(q-k)$-face $\tau$ of a simplex $\sigma$  we have $int(\tau)\cap E=\emptyset$, i.e. $\tau$ is not supported on $E$ whenever $k+l<n$.
In addition, $\gamma'$ satisfies to the following properties:
\begin{itemize}
\item
$\gamma \cap (X\setminus(U_\alpha \cap \ostar(E))) = \gamma'\cap (X\setminus(U_\alpha \cap \ostar(E)))$,
 \item
 the chain $\gamma-\gamma'$ is the boundary of a tropical $(q+1)$-chain supported in $U_\alpha \cap\ostar(E)$,
 \item
 if $E$ has positive sedentarity then any simplex $\sigma$ such that $\sigma \cap E \ne \emptyset$ has its coefficient $\beta_\sigma$ divisible by all divisorial vectors corresponding to $E$.
\end{itemize}
\end{lemma}

\begin{proof}
First let us consider the case when $E$ is mobile. Working in a chart we can assume $Y$ is the Bergman fan
for some loopless matroid $M$.
Clearly, any matroid $M$ contains a uniform submatroid $M_0\subset M$ of the same rank $r(M)$
(by submatroid we mean a subset with the restriction of the rank function).
Thus, we have a sequence $M_0 \subset \dots \subset M_{|M|-r(M)}=M$ of submatroids of $M$
such that $M_{j+1}$ is obtained from $M_{j}$ by adding one element $\epsilon_{j+1}$. 
We may form a matroid $H_j$ of rank $r(M_j)-1$ by setting a new rank function $r_{H_j}$ on $M_j$,
$r_{H_j}(A)=r_M(A\cup \epsilon_{j+1})-1$ for $A\subset M_j$.

The fan $Y_{M_{j+1}}\subset \R^{|M_j|}$ maps to the fan $Y_{M_j}\subset \R^{|M_j|-1}$
by projection along the coordinate corresponding to the element $\epsilon_{j+1}$. 
If the matroid $H_j$ has loops
this map 
$$\tau_j:Y_{M_{j+1}}\to Y_{M_j}
$$ 
is an isomorphism.
\begin{figure}
\centering
\includegraphics[height=60mm]{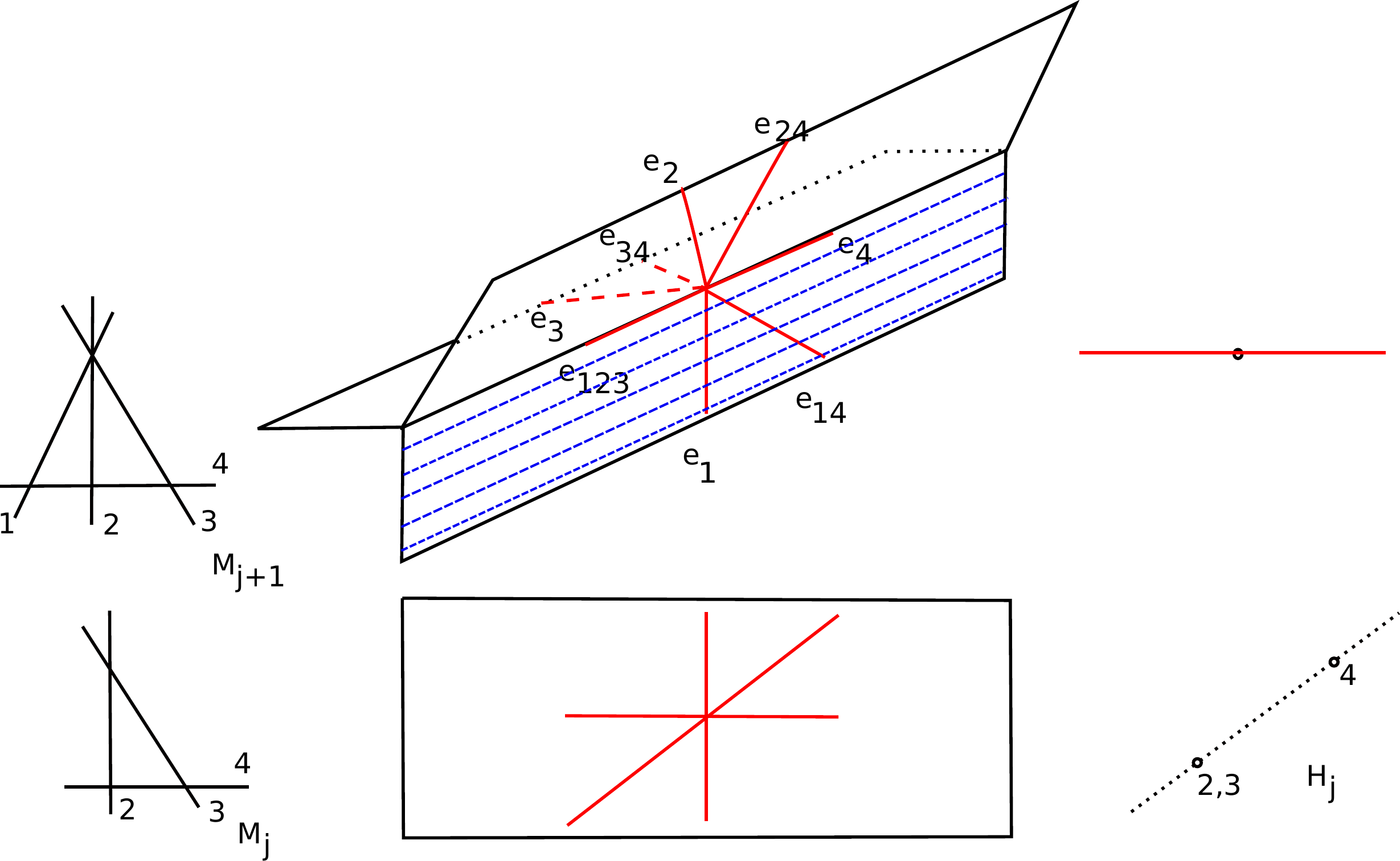}
\caption{The matroids $M_{j+1}, M_j$ and $H_j$ and their corresponding fans. The fan $Y'_{M_{j+1}}$ is the unshaded part of  $Y_{M_{j+1}}$. The shaded part of  $Y_{M_{j+1}}$ is $\ostar(e_{\epsilon_{j+1}})$.}
\label{fig:lemma59} 
\end{figure}
Otherwise note that the Bergman fan $Y_{H_j}$ is a subfan of $Y_{M_j}$. Also we denote by $Y'_{M_{j+1}}$ the subfan of $Y_{M_{j+1}}$ containing only those cones whose corresponding flags do {\em not} have two flats differing just by $\epsilon_{j+1}$, see Fig. \ref{fig:lemma59}. Then $\tau_j: Y'_{M_{j+1}}\to Y_{M_j}$ is a one-to-one map linear on the cones, cf. \cite{Shaw}.
Indeed, $\tau_j$ contracts precisely those cones of $Y_{M_{j+1}}$ which are parallel to
$e_{\epsilon_{j+1}}$. 

$Y_{M_0}$ is a complete fan in $\R^{r(M)-1}$ and the lemma is trivial since the coefficients $\F_p = \Lambda^p \Z^{r(M)-1}$ are constant on all strata and we may deform $\gamma$ into a general position (subdividing simplices in $\gamma$ if needed to keep the chain strata-compatible).
Inductively we suppose that the lemma holds for $Y_{M_j}$ and the matroid $H_j$ is loopless
and then prove that the lemma holds for $Y_{M_{j+1}}$.

We denote by $\ostar(e_{\epsilon_{j+1}})$ the complement of $Y'_{M_{j+1}}$ in $Y_{M_{j+1}}$. It really is the open star of $e_{\epsilon_{j+1}}$ (in the coarsest face structure of $Y_{M_{j+1}}$). Note that $\ostar(e_{\epsilon_{j+1}})\cong Y_{H_j}\times\R$.

If $E\subset\ostar(e_{\epsilon_{j+1}})$ we may use the inductive
assumption for projections to $Y_{H_j}$ (it has smaller dimension) together with a deformation 
along a generic vector field parallel to $e_{\epsilon_{j+1}}$.

If $E\not\subset\ostar(e_{\epsilon_{j+1}})$, that is $E$ is contained in $Y'_{M_{j+1}}$ we have $\dim(\tau_j(E))=\dim(E)=l$. Consider singular $q$-simplices from $\gamma$ with the interiors mapped to $\ostar(E)$ and such that their closures intersect $E$. These simplices form a chain $\gamma_E$ which can be considered as a relative cycle modulo its boundary $\dd \gamma_E$. We have $\dd \gamma_E\cap E=\emptyset$. Furthermore, $\tau_j(\dd \gamma_E)$
is a $(q-1)$-cycle in the $(n-1)$-dimensional tropical manifold $Y_{H_j}$.
By induction on dimension we may assume that $\tau_j(\dd \gamma_E)\cap \ostar(e_{\epsilon_{j+1}})$ can be deformed in $Y_{H_j}$
to a cycle with simplices
without faces of dimension larger than $q-n+l$ whose relative interiors are contained in $E$.
As $\ostar(e_{\epsilon_{j+1}})\cong Y_{H_j}\times\R$ such deformation lifts to $Y_{M_{j+1}}$ and can be extended to 
a deformation of $\gamma$ in $Y_{M_{j+1}}$.  

By induction on $j$ 
there exists a tropical chain $b_j\in C_{q+1}(Y_{M_j};\F_p)$ such that the relative interiors of $k$-faces of 
singular simplices of  $\gamma'_j=\dd B_j - \tau_j (\gamma)$ are disjoint from $E$.
This assumption holds for any face structure on $Y_{M_j}$, in particular for the one compatible with $Y_{H_j}$.
Then the relative interiors of all $q$-dimensional simplices are disjoint from $Y_{H_j}$ and we
can form $\tilde b_j\in C_{q+1}(Y_{M_{j+1}};\F_p)$ and $\tilde \gamma'_j\in C_{q+1}(Y_{M_{j+1}};\F_p)$
by applying $\tau_j^{-1}|_{Y_{M_j}\setminus Y_{H_j}}$ to $b_j$ and $\gamma'_j$.
Note that $\dd\tilde b_j - \gamma - \tilde \gamma'_j$ must have 
the coefficients 
vanishing under $\tau_j$, even though generated from the facets of $Y'_{M_{j+1}}$.
Such coefficients must be supported on $\ostar(e_{\epsilon_{j+1}})$ and thus we may apply
the same reasoning as in the case of $E\subset\ostar(e_{\epsilon_{j+1}})$.

Finally, let us now consider the sedentary case, that is let $E$ be a sedentarity $s$ face of $Y_M\times\T^s$ with $s=|I|>0$. Let $\xi_j$ be the divisorial vectors, and let $V_J:=\wedge_{j\in J} \xi_j$ denote the divisorial $|J|$-polyvector for each $J\subset I$.
We will need to deform $\gamma$ to $\gamma'$ so that no $(q-s)$, or smaller, -dimensional face of a simplex $\sigma$ in $\gamma'$ meets $Y_M\times\{-\infty\}$ (here $\{-\infty\}\in\T^s$ is the point of
sedentarity $s$).
In $Y_M\times\R^I$ the groups $\F_p$ split into the  direct sum $\oplus_{J\subset I} \F_p^J$,  where $\F_p^J$ consists of elements divisible by the polyvector $V_J$, and no larger $V_{J'}$. (The splitting is not canonical, it depends on a chart). Accordingly, we have a decomposition $\gamma=\sum_{J\subset I} \gamma_J$ into cycles.

If $J\ne I$, that is there exists $j\notin J$, we may push $\gamma_J$ from $E$ with the help of a vector field parallel to $x_j$. Note that $\gamma_J$ remains a cycle after such deformation as $\xi_{j}$ is not present in the coefficients of $\gamma_J$. Thus by induction on sedentarity we may assume $J=I$.

The cycle $\gamma_I$ has coefficients in $\F_{p-s}^{Y_M}\otimes V_I$, and hence can be interpreted as
a relative cycle modulo $\dd\T^I=\T^I\setminus\R^I$ with coefficients in $\F_{p-s}^{Y^M}$
(as $V_I$ vanishes on $\dd\T^I$ and constant otherwise) and $(T^I,\dd T_I)$ is homeomorphic
to the pair $\R^{s-1}\times (\R_{\ge 0},\{0\})$ of a half-space and its boundary.
Thus $\gamma_I$ may be deformed to a product (after simplicial subdivision) 
of the relative fundamental cycle in the $s$-dimensional half-space with
some $(q-s)$-dimensional singular cycle. In particular, $E$ will not meet any codimension $<s$ face of a  $q$-simplex in a deformed cycle. 
\end{proof}

\begin{remark}
Let $\Sigma = \bigcup \sigma$ be an integral polyhedral fan (with its cones $\sigma$ oriented). Then using the inclusion homomorphisms \eqref{cosheafmap} we can form the complex $C^{(p)}_k:= \oplus_{\dim \sigma =k} \F_p(\sigma)$. In case $\Sigma$ is a matroidal fan the statement of Lemma \ref{transversal-E} is equivalent to that the complex  $C^{(p)}_\bullet$ has only the highest homology.
\end{remark}

\begin{remark}
When $X$ is not smooth the statement of the Lemma is not true. For example let $X$ be a union of two 2-planes in $\R^4$ intersecting in a point. Consider an unframed path (that is cycle in $C_1(X;\F_0)$ through the vertex which starts in one plane and ends in the other plane. Any deformation of this path will still have to go through the vertex. 
\end{remark}

\begin{corollary}\label{lemma:transversal} 
Let $X$ be a tropical manifold. Then
\begin{enumerate}
\item Every class in $H_q(X; \F_p)$ is represented by a transversal cycle.
\item Every pair of classes in $H_{q'}(X; \F_{p'})$ and $H_{q''}(X; \F_{p''})$ is represented by a transversal pair of cycles.
\item If $\gamma'_1, \gamma'_2$ are two cycles which represent the same class in  $H_{q'}(X; \F_{p'})$ and both form transversal pairs with a cycle $\gamma''\in C_{q''}(X; \F_{p''})$,
then there is $b\in C_{q'+1}(X; \F_{p'})$ which form a transversal pair with $\gamma''$, and such that $\partial b=\gamma'_1-\gamma'_2$.
\end{enumerate}
\end{corollary}
\begin{proof}
We may start from any tropical cycle and deform it to a transversal position by applying
Lemma \ref{transversal-E} stratum by stratum starting from $0$-dimensional faces and then higher-dimensional strata.
(Note that in a chart the open star of any face can intersect only faces of higher dimension).

Suppose that we have two transversal cycles. Since any stratum $\s$ is a manifold we can make interiors of faces of the simplices from these cycles transversal in $\s$ by a small deformation
with the help of the usual Sard's theorem. In any chart this deformation extends to a small deformation in $\ostar(\s)$.
Making this procedure stratum by stratum in the order of non-decreasing dimension we make any pair of cycles transversal. A similar argument applies to the relative cycle in the last statement of the corollary.
\end{proof}


If $p'+p''+q'+q''=2n$ we can give a numerical value to the product $\gamma'\cdot\gamma''$ by integrating the $(n-p'-p'')$-form $\Omega_{\s_\tau} (\beta_{\sigma'}\wedge \beta_{\sigma''})$ over the $(q'+q''-n)$-simplex $\tau$. Indeed, since $\beta_{\sigma'}\wedge \beta_{\sigma''}$ is divisible by all divisorial directions corresponding to sedentary faces of $\tau=\sigma'\cap \sigma''$, the integration can be carried over in the quotient space to those (infinite) coordinates, thus giving a finite answer. Thus we define
\begin{equation}\label{eq:pairing}
\int \gamma' \cdot \gamma'':=\sum_
{\tau \subset \sigma'\cap \sigma''}
\int_\tau\Omega_{\s_\tau}(\beta_{\sigma'}\wedge \beta_{\sigma''}) \in \R.
\end{equation}

The most interesting case to us is when $p'+q'=p''+q''=n$. Assuming $q'+q''\ge n$ we can use the eigenwave action on one of the cycles in the pair to make them of complementary dimensions, after which the integration becomes just summing over the intersection points
$$\<\gamma',\gamma''\>:=\int \gamma' \cdot \gamma''= \sum_{x\in |\gamma'|\cap|\gamma''|} \Omega_x (\beta'_x\wedge\beta''_x),
$$
where $\beta'_x, \beta''_x$ are the coefficients at $\sigma',\sigma''$ for their intersection points $x\in\sigma'\cap \sigma''$.

\begin{proposition}\label{prop:wave_commute}
Let $\gamma'=\sum\beta_{\sigma'} \sigma'\in C_{q'}(X; \F_{p'})$ and $\gamma''=\sum\beta_{\sigma''} \sigma'' \in C_{q''}(X; \F_{p''})$ be a transversal pair of tropical cycles with $p'+q'=p''+q''=n$ and $q'+q''\ge n$. Let $k:=q'-p''= q''-p' \ge0$.  Then 
$$ \<\phi^k \cap \gamma', \gamma''\>=\int \gamma'\cdot\gamma''.
$$
\end{proposition}
\begin{proof}
First we need a representative of the cycle $\phi^k \cap \gamma'$ such that it still forms a transversal pair with $\gamma''$. We fix first and second baricentric subdivisions of the simplices $\sigma'$ in $\gamma'$. Then by transversality of $\gamma''$ we can assume that the intersection of each $\sigma'$ with $\gamma''$ is supported on the star skeleton of $\sigma'$. That is $\sigma'\cap |\gamma''|$ consists of the $k$-simplices of the first baricentric subdivision of $\sigma'$ spanned by the baricenters of the $q'-k, \dots, q'$-dimensional faces $\tau$ of $\sigma'$. We label the $k$-simplices in the first baricentric subdivision of $\sigma'$ by the flags of its faces $(\tau_0\prec\dots \prec \tau_k)$.

\begin{figure}
\centering
\includegraphics[height=30mm]{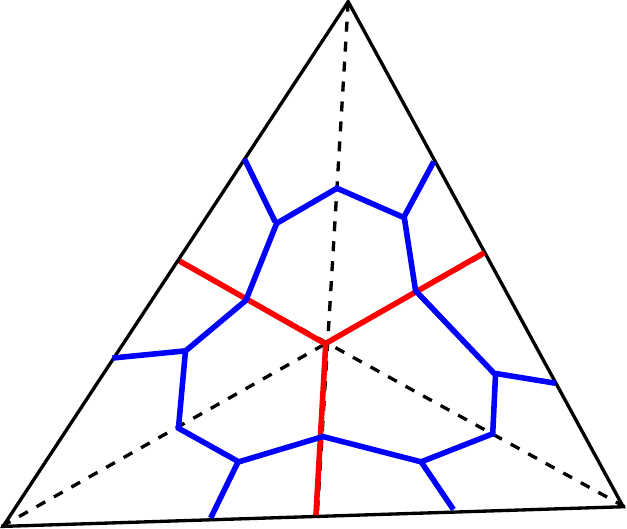}
\caption{Intersection in $\sigma'$: $ |\gamma''|$ (in red), $|\phi^k \cap \gamma'|$ (in blue).}
\label{fig:intersection} 
\end{figure}

Then the result of the wave action \eqref{eq:wave_action1} from Description 1 on $\beta_{\sigma'}\sigma'$ gives the following chain (see Fig. \ref{fig:intersection}) 
$$\sum_{\tau_0\prec\dots \prec \tau_k} (w_{\tau_0\prec\dots \prec \tau_k}\wedge\beta_{\sigma'}) \widehat{(\tau_0\prec\dots \prec \tau_k)},
$$ 
where $w_{\tau_0\prec\dots \prec \tau_k}\in W_k (\Delta_{\sigma'})$ is the polyvector associated to the simplex  $(\tau_0\prec\dots \prec \tau_k)$, and $\widehat{(\tau_0\prec\dots \prec \tau_k)}$ is its star dual in the second baricentric subdivision (cf. definition in Section \ref{subsection:wave_action}).
When intersected with $\gamma''$ only the simplices $(\tau_0\prec\dots \prec \tau_k)$ with maximal dimensional flags enter and we see that the result coincides with the definition of $\int \gamma' \cdot \gamma''$.
\end{proof}

\begin{proposition}\label{prop:product_homology}
Let $X$ be smooth. Then the intersection product $\<\ , \ \>$ on cycles descends to a pairing on homology $H_q(X;\F_p) \otimes H_{p}(X; \F_{q}) \to \R$.
\end{proposition}

\begin{proof}
Suppose that we have two homologous cycles $\gamma'_1\in C_q(X;\F_p)$ and $\gamma'_2 \in C_q(X;\F_p)$. Let $b\in C_{q+1}(X;\F_p)$ be the connecting chain, i.e. $\dd b=\gamma'_1-\gamma'_2$. According to Corollary \ref{lemma:transversal} we can assume that each of the three $\gamma_1', \gamma_2', b$ forms a transversal pair with a cycle $\gamma'' \in C_p(X; \F_q)$. 

It is clear that $\dd (b \cdot \gamma'')$ coincides with the $\gamma_1' \cdot \gamma'' - \gamma_2' \cdot \gamma''$ on the interiors of the maximal strata of $X$. Thus it is enough to show that $b \cdot \gamma''$ has no boundary on codimension 1 mobile strata of $X$ (according to Lemma \ref{transversal-E} the intersection has no support on infinite simplices). This is local so we can work in a chart $\phi_\alpha: U_\alpha \to Y \subset \T^N$.

Let $E$ be a codimension 1 face of $Y$ and let  $D_1,\dots,D_k$ be the adjacent facets at $E$. We choose $v_1,\dots,v_k$, the corresponding primitive vectors such that  $\sum_{i=1}^k v_i=0$ (not just modulo the span of $E$). Let $x$ be a point in the relative interior of $E$ where $b$ intersects $\gamma''$, and let $\tau_1, \dots, \tau_k$ be the intervals in the support of $b \cdot \gamma''$ adjacent to $x$. Each $\tau_i$ lies in $D_i$. Let $\beta_i'\in \F_p(D_i)$ and  $\beta_i''\in \F_q(D_i)$ be the coefficients of the simplices of $b$ and of $\gamma''$, respectively, which intersect at the $\tau_i$.

Since $\gamma''$ is a cycle, we have $\sum_i \beta_i''=0$. We can write each 
$$\beta_i''=v_i\wedge \bar \alpha_i'' + \alpha_i'',
$$ 
where $\bar\alpha_i'' \in W_{q-1}(E)$ and $\alpha_i'' \in W_q (E)$. 

Recall that our tropical space $X$ is smooth. In particular, this means that the fan at $E$ modulo linear span of $E$ is matroidal. That is, $\sum_{i=1}^k v_i=0$ is the {\em only} linear relation among the $v_i$'s. This together with $\sum_i \beta_i''=0$  implies that 
$$\sum_{i=1}^k \alpha_i''=0 \quad \text{ and } \quad \bar\alpha_1''=\dots=\bar\alpha_k'' =: \bar\alpha''.
$$

Similarly, $\sum_i \beta_i'=0$ since $\dd b$ cannot have support at $x$. Hence we can write 
$$\beta_i'=v_i\wedge \bar \alpha' + \alpha_i',
$$ 
with $\sum \alpha_i'=0$,  $\alpha_i' \in W_p (E)$ and $\bar\alpha' \in W_{p-1}(E)$. Note that in the product 
\begin{multline*}
\beta_i'\wedge\beta_i''=(v_i\wedge \bar \alpha' + \alpha_i')\wedge (v_i\wedge \bar \alpha'' + \alpha_i'')
=v_i\wedge(\bar\alpha'\wedge \alpha_i'' + \alpha_i'\wedge \bar \alpha'')
\end{multline*}
only the cross terms survive. Now we are ready to evaluate $\dd (b\cdot \gamma'')$ at $x$:
\begin{multline*}
 \sum_i \Omega_{\Delta_i} [ v_i\wedge(\bar\alpha'\wedge \alpha_i'' + \alpha_i'\wedge \bar \alpha'') ] 
= \sum_i \Omega_{\Delta} (\bar\alpha'\wedge \alpha_i'' + \alpha_i'\wedge \bar \alpha'')  \\
=  \Omega_{\Delta} (\bar\alpha'\wedge \sum_i  \alpha_i'' +\sum_i  \alpha_i'\wedge \bar \alpha'')  = 0.
\end{multline*}
\end{proof}

Finally we restrict to the case when both $\gamma', \gamma''$ are cycles in  $C_{q}(X; \F_{p})$ with $p+q=n$. Then $\gamma'\cdot \gamma'' = \gamma'' \cdot \gamma'$. Indeed, assuming the orientation of $\tau$ is chosen, taking the product in the opposite order will result in the change of sign of the volume form $\Omega_{\s_\tau}$ according to the parity of $p$. On the other hand this parity will also affect the coefficients product: $\beta'\wedge\beta''= (-1)^p \beta''\wedge\beta'$, both effects cancel in $\Omega_{\s_\tau}(\beta'\wedge\beta'')$.
This observation combined with Propositions \ref{prop:wave_commute} and \ref{prop:product_homology} lead to the final statement.
 
\begin{theorem}\label{theorem:intersection}
Let $X$ be compact and smooth. The product on cycles \eqref{eq:pairing} descends to a symmetric bilinear form on $H_q(X; \F_p)$ for any $p+q=n$. 
\end{theorem}

\begin{conjecture}\label{conj:non_degenerate}
This form is non-degenerate.
\end{conjecture}

\section{Appendix: Konstruktor and the eigenwave action in the realizable case}\label{section:konstruktor}

\subsection{Tropical limit and the Steenbrink-Illusie spectral sequence}
Suppose $X$ is the tropical limit of a complex projective one-parameter degeneration $\mathcal X \to \Delta^*$. Then $X$ is naturally polyhedral. We assume also that $X$ is smooth.
In this case the refined stable reduction theorem \cite{Mumf} allows us assume the following (see details in \cite{IKMZ}).
\begin{itemize}
\item $X$ is unimodularly triangulated. This means that the finite cells are unimodular simplices and the infinite cells are products of unimodular simplices and unimodular cones spanned by the divisorial vectors.
\item The finite part of $X$ is identified with the dual Clemens complex of the degeneration with simple normal crossing central fiber $Z=\cup Z_{\alpha}$. This means that the components of $Z$ are labelled by vertices of zero sedentarity  and their intersections $Z_{\alpha_0} \cap\dots \cap Z_{\alpha_k}=:Z_\Delta$ are labelled by (finite) simplices $\Delta=\{\alpha_0\dots\alpha_k\}$ of $X$ of zero sedentarity.
\end{itemize}

\begin{theorem}\label{theorem:isomorphism}
Let $X$ be a realizable smooth projective tropical variety. Then for $q\ge p$
$$ \phi^{q-p}:  H_q(X; \F_p) \otimes\Q \to H_{p}(X; \F_{q})  \otimes\Q
$$
is an isomorphism.
\end{theorem}

In this algebraic setting the eigenwave itself is an integral class in $H^1(X;\W_1)$ (recall that $W$ carries a natural lattice). Hence in the statement we can avoid tensoring the tropical homology groups with $\R$. However its proof relies on the isomorphism in Theorem \ref{theorem:main} which we can assert only over $\Q$. Although we believe that the theorem remains true over $\Z$ its proof may be more delicate.

We will prove the theorem by comparing the eigenwave action with the classical monodromy action $T: H_k(X_t, \Q) \to H_k(X_t, \Q)$, where $X_t$ is a general fiber in $\mathcal X$. The idea that the monodromy can be represented by a cap product with certain cohomology class appeared before in the Calabi-Yau case. The second author \cite{Zh00} proved a related conjecture of Gross \cite{Gr98} that for toric hypersurfaces the monodromy can be described as the fiber-wise rotation by a natural section of the SYZ fibration. Later Gross and Siebert (\cite{GS10}, Section 5.1) explored the relation between the monodromy and the cap product in the logarithmic setting.

Notations:
\begin{itemize}
\item $\Delta$ or $\Delta'$ will always denote a finite face of $X$ of sedentarity 0, in particular, a simplex.
\item $ H_{2l}(\Delta)[-r] = H_{2l}(Z_{\Delta},\Q)$, Tate twisted by $[-r,-r]$.
\item $H_{2l}(k)[-r] = \oplus H_{2l}(\Delta)[-r]$, where $\Delta$ runs over all $k$-simplices in $X$ as above.
\end{itemize}

First  we recall the classical spectral sequence which calculates the limiting mixed Hodge structure of the family  $\mathcal X$ (see, e.g. \cite{Steen}, Chapter 11). This spectral sequence (from now on referred to as the Steenbrink-Illusie's, or SI for short) has the first term
$${E}^1_{r, k-r}=\bigoplus_{i\ge \max\{0, r\}} H_{k+r-2i}(2i-r)[r-i],
$$
and it degenerates at $E_2$ abutting to homology of the smooth fiber $X_t$ of $\mathcal Z$ with the monodromy weight filtration. 

Since all strata in $Z$ are blow ups of projective spaces,  the odd rows in Steenbrink-Illusie's $E^1$ vanish. Removing those and making shifts in the even rows we relabel the terms by 
$$\tilde{E}^1_{q,p}:={E}^1_{q-p, 2p}=\bigoplus_{i\ge \max\{0, q-p\}} H_{2q-2i}({2i+p-q})[q-p-i].
$$
The first differential $d=d'+d''$ consists of the map $d'$ induced by strata inclusion and the Gysin map $d''$: 
\begin{equation*}
\begin{split}
& d' :H_{2l}({k})[-r] \to H_{2l}({k-1})[-r] \\
& d'' : H_{2l}({k})[-r] \to H_{2l-2}({k+1})[-r-1]. 
\end{split}
\end{equation*}

For reader's convenience we write the beginning of the $\tilde{E}^1$ term:
$$ \xymatrix{
 H_0(4)[-4] &  & & &
 \\
 H_0(3)[-3] & 
 \txt{$H_0(4) [-3]$ \\ $\oplus H_2(2)[-2]$}  \ar[l]_{d} \ar[ul]_\nu & & &
  \\ 
 H_0(2)[-2] &
 \txt{$H_0(3)[-2] $ \\ $\oplus H_2(1)[-1]$}  \ar[l]_{d} \ar[ul]_\nu  &
 \txt{$H_0(4)[-2] $ \\ $\oplus H_2(2) [-1]$ \\ $\oplus H_4(0) $}  \ar[l]_{d} \ar[ul]_\nu &  & 
\\
H_0(1)[-1] & 
 \txt{$H_0(2) [-1]$ \\ $\oplus H_2(0)$}  \ar[l]_{d} \ar[ul]_\nu &
 \txt{$H_0(3) [-1]$ \\ $\oplus H_2(1)$}  \ar[l]_{d} \ar[ul]_\nu &
  \txt{$H_0(4) [-1]$ \\ $\oplus H_2(2)$}  \ar[l]_{d} \ar[ul]_\nu & 
\\
H_0(0) & H_0(1) \ar[l]_{d} \ar[ul]_\nu  & H_0(2) \ar[l]_{d} \ar[ul]_\nu  & H_0(3)  \ar[l]_{d} \ar[ul]_\nu & H_0(4) \ar[l]_{d} \ar[ul]_\nu 
 }
$$
The monodromy operator $\nu=\frac1{2\pi i} \log T$ acts along the diagonals by the Tate twist isomorphism $H_{2l}(k)[-r] \to H_{2l}(k)[-r-1]$ or by 0 if the corresponding group is missing (cf. \cite{Steen}, Chapter 11).

\subsection{Propellers}
Next we will give a combinatorial description of the SI groups and the differential in terms of  {\em propellers} - the ``local tropical cycles'' in $X$.

Some more notations: 
\begin{itemize}
\item Recall that $\Delta, \Delta', \Delta''$ always denote finite faces of $X$ of sedentarity 0.
\item
We write $\Delta\prec_k\Delta'$ or $\Delta'\succ_k\Delta$, if $\Delta$ is a face of $\Delta'$ of codimension $k$.

\item $\Link_l(\Delta)$ consists of sets $\bar q=\{q_1,\dots,q_l\}$ where each $q_i$ is either a vertex or a divisorial vector, such that the vertices of $\Delta$ together with elements of $\bar q$ span a face (infinite, in case $\bar q$ contains divisorial vectors) adjacent to $\Delta$ of dimension $l$ higher. We denote the corresponding face by $\{\Delta \bar q\}$ and 
often drop the brackets from the notation (e.g., as below) when they become cumbersome.

\item $\Link_l^0(\Delta) \subset \Link_l(\Delta)$ consists of those sets $\bar q=\{q_1,\dots,q_l\}$ where $q_i$ are allowed to be only vertices (not the divisorial vectors). In this case $\{\Delta \bar q\}$ is finite.

\item $\Vol_{\Delta\bar q}$ is the integral volume element in the (oriented) face $\{\Delta \bar q\}$.
\end{itemize}
 
Let $\Delta$ be an oriented finite cell of sedentarity 0. One can naturally identify (see \cite{IKMZ} for details) the homology groups $H_{2l}(\Delta)$ with the space of local tropical relative $l$-cycles around $\Delta$. That is, we consider formal $\Q$-linear combinations 
$$\sum_{\bar q \in \Link_l(\Delta)} \rho_{\bar q} \{\Delta {\bar q}\}
$$
 of (possibly infinite) cells $\{\Delta {\bar q}\}\succ_l\Delta$ which are balanced along $\Delta$. We call these local cycles {\em propellers} and abusing the notation we continue denoting this group by $H_{2l}(\Delta)$ (there is no Tate twist however).

Then one can identify the Gysin map $d'': H_{2l}(\Delta) \to H_{2l-2}(\Delta')$ with the restriction of the propeller to a consistently oriented finite simplex $\Delta'\succ_1 \Delta$. Put together 
\begin{equation}\label{eq:d''}
d'' ( \sum_{\bar q \in \Link_l(\Delta)} \rho_{\bar q} \{\Delta {\bar q}\}  )=  \sum_{q \in \Link_1^0(\Delta)} ( \sum_{\bar r \in \Link_{l-1} (\Delta q)} \rho_{q \bar r} \{\Delta {q \bar r}\} ).
\end{equation}

The inclusion map $d': H_{2l}(\Delta) \to H_{2l}(\Delta')$, where $\Delta'= \Delta \setminus v$ is consistently oriented facet of $\Delta$, is somewhat more tricky. Let 
$c=\sum_{\bar q \in \Link_l(\Delta)} \rho_{\bar q}\{\Delta{\bar q}\}
$ 
be an element in $H_{2l}(\Delta)$.
For any ${\bar q \in \Link_l(\Delta)}$ let $\{\Delta'{\bar q}\}=\{\Delta{\bar q}\setminus v\}$ be the corresponding cell containing $\Delta'$. 
Then the image of $ d' c$ in $H_{2l}(\Delta')$ will be 
\begin{equation}\label{eq:d'}
\sum_{\bar q \in \Link_l(\Delta)} \rho_{\bar q} \{\Delta'{\bar q}\} + \sum_{\bar r \in \Link_{l-1}(\Delta)} \rho_{v \bar r} \{\Delta{\bar r}\},
\end{equation}
where the coefficients $\rho_{v \bar r}\in \Q$ are chosen to make the result balanced along $\Delta'$. There is always a unique such choice (cf. \cite{IKMZ}), namely, the $\rho_{v \bar r}$ can be read off from the balancing condition for $c$ along $\{\Delta{\bar r}\}$:
\begin{equation}\label{eq:descent}
\sum_q \rho_{q \bar r} \overrightarrow{(\Delta' q)} + \rho_{v \bar r} \overrightarrow{(\Delta' v)} =0 \quad \mod \{\Delta' \bar r\},
\end{equation}
where $\overrightarrow{(\Delta' q)}$ means the divisorial vector $q$, or the vector from any vertex of $\Delta'$ to $q$ (well defined mod $\Delta'$) if $q$ is a vertex, and same for  $\overrightarrow{(\Delta' v)}$.

From now on we will not distinguish between the classical geometric Steenbrink-Illusie $E_1$ complex and its interpretation via complex of propellers. One of the main results in \cite{IKMZ} is the following statement.

\begin{theorem}[\cite{IKMZ}]\label{theorem:main}
$\tilde{E}^2_{q,p}\cong H_{q}(X;\F_p) \otimes \Q$.
\end{theorem}

\subsection{Konstruktor}
Now we provide another realization of the Steenbrink-Illusie's $E_1$ complex in terms of specific tropical simplicial chains. The collection of these chains which we call {\em konstruktor} forms a subcomplex of  $C_\bullet^{bar} (X, \F_\bullet)$, and we can refer to Theorem \ref{theorem:main} to see that the inclusion is a quasi-isomorphism. A wonderful feature of the konstruktor is that the eigenwave acts on its elements precisely as the monodromy operator $\nu$ acts on the terms in the Steenbrink-Illusie's $E_1$.

Let us fix the first baricentric subdivision of $X$. We elaborate a little bit on already used notation of the dual cell.
\begin{itemize}
\item For a pair $\Delta\succ \Delta'$ of finite simplices of sedentarity 0 in $X$, and $\bar q \in \Link_l(\Delta)$
we let $ \hat \Delta'_{\Delta \bar q}$ denote the dual cell to $\Delta'$ in the face $\{\Delta \bar q\}$ of $X$, that is 
the union of all simplices in the baricentric subdivision containing baricenters of both $\Delta'$ and  $\{\Delta{\bar q}\}$.
\item In the summation formulae to follow we assume the terms with $ \hat \Delta'_{\Delta \bar q}$ are not present if $\Delta'$ is not a zero sedentarity finite face of $\{\Delta \bar q\}$.
\end{itemize}

Let $\Delta$ be a finite $k$-simplex of sedentarity 0 in $X$, and $r\le k$ a non-negative integer.
To any propeller, that is a local tropical $l$-cycle 
$$c=\sum_{\bar q \in \Link_l(\Delta)} \rho_{\bar q} \{\Delta{\bar q}\} \in H_{2l}(\Delta)
$$ 
we associate a simplicial chain $c[-r] \in C^{bar}_{k+l-r}(X,\F_{l+r})$ as follows (note that $c[0]$ now has other meaning than just $c$):
$$c[-r] =\sum_{\bar q \in \Link_{l}(\Delta)}
\sum_{
\begin{subarray}{c}
\Delta'\prec \Delta\\
\dim \Delta'= r
\end{subarray} 
} 
(\rho_{\bar q}  \Vol_{\Delta'\bar q}) \hat \Delta'_{\Delta \bar q} .
$$
The orientation of $\hat \Delta'_{\Delta \bar q}$ is consistent with the original orientation of $\Delta$ and the choice of the volume element $\Vol_{\Delta'\bar q}$.  
Clearly for each $r$ between 0 and $k$ the map
$$ (\cdot) [-r]: H_{2l}(k) \to C^{bar}_{k+l-r}(X,\F_{l+r})
$$
is an injective group homomorphism. We denote its image in $C^{bar}_{k+l-r}(X,\F_{l+r})$ by $K_l(k)[-r]$. 

\begin{definition}
The {\em konstruktor} is the subgroup of $C^{bar}_{\bullet}(X,\F_{\bullet})$ generated by the $K_l(k)[-r]$ for all $k$, $l$ and $r$ . Note that $K_l(k)[-r]$ intersect trivially for different triples $k,l,r$.
\end{definition}

Next we want to show that for each $p$ the $\oplus_r K_{p-r}(\bullet -p+2r)[-r]$ is indeed a subcomplex of  $C^{bar}_\bullet(X,\F_p)$ isomorphic to the SI complex $\tilde{E}_1^{\bullet,p}$. This follows at once from comparing the SI differentials $d=d'+d''$ with the simplicial boundary $\dd$.

\begin{proposition}\label{prop:konstruktor}
$ \dd (c[-r]) = (d'c)[-r] + (d''c)[-r-1]$.
\end{proposition}

\begin{proof}
For the proof we need two linear algebra identities. Let $\sigma',\sigma''$ be two opposite faces in a unimodular simplex $\sigma=\{\sigma' \sigma''\}$. Then one has
$$\sum_{\tau''\prec_1\sigma''} \Vol_{\sigma'\tau''}=\Vol_{\sigma'} \wedge \Vol_{\sigma''}=\sum_{\tau'\prec_1\sigma'} \Vol_{\tau'\sigma''},
$$ 
where, say, the left equality easily follows from the case when $\sigma'$ is a vertex. Here all $\tau'$ are oriented consistently with $\sigma'$, and all $\tau''$ with $\sigma''$. We will need this identity in the form
\begin{equation}\label{eq:simplex}
\sum_{\Delta'\prec_1\Delta} \Vol_{\Delta'\bar q}= \sum_{q\in \Link_1^0(\Delta)} \Vol_{\Delta \bar q\setminus q},
\end{equation}
where $\Delta$ is a finite simplex and $\bar q\in \Link_l(\Delta)$. Note that  the divisorial vectors (if any) in $\bar q$ just multiply both sides of the identity for finite simplices.

The second identity involves a relation among the balancing coefficients $\rho_{v \bar r}$ from  \eqref{eq:d'} for $c=\sum \rho_{\bar q} \{\Delta{\bar q}\}$. One can show (cf. \cite{IKMZ}) that they satisfy a refined version of \eqref{eq:descent}. Namely, for $\Delta'\prec \Delta \prec \{\Delta \bar q\}$ we have
$$
\sum_q \rho_{q \bar r} \overrightarrow{(\Delta' q)} +\sum_{v\in \Delta\setminus \Delta'} \rho_{v \bar r} \overrightarrow{(\Delta' v)} = 0 \quad  \mod \{\Delta' \bar r\}
$$
for faces $\Delta'\prec\Delta$ of codimension possibly higher than 1.
Multiplying the above by $\Vol_{\Delta'\bar r}$ we arrive at 
\begin{equation}\label{eq:vol_balance}
\sum_{q\in \Link_1(\Delta)} \rho_{q \bar r} \Vol_{\Delta' q \bar r} = - \sum_{v \in \Delta\setminus \Delta'} \rho_{v \bar r} \Vol_{\Delta' v \bar r}.
\end{equation}

Now we are ready to proof the proposition. Let $c=\sum \rho_{\bar q} \{\Delta{\bar q}\}$, then we can write
$$c[-r] =\sum_{
\begin{subarray}{c}
\Delta'\prec_{k-r} \Delta\\
\bar q \in \Link_{l}(\Delta)
\end{subarray} 
} 
(\rho_{\bar q} \Vol_{\Delta'\bar q}) \hat \Delta'_{\Delta{\bar q}}  .
$$
The topological boundary of each cell $\Delta'_{\Delta{\bar q}}$ consists of two types:
\begin{itemize}
\item 
Type 1: cells in the form $\Delta''_{\Delta{\bar q}}$ for faces $\Delta''\succ_1\Delta'$ of  $\{\Delta \bar q\}$. If the cell $\Delta'_{\Delta{\bar q}}$ includes divisorial directions then its coefficient $\Vol_{\Delta'\bar q}$ in $c[-r]$ is divisible by all divisorial vectors. Hence the type 1 part of the boundary $\dd (c[-r])$ is, in fact, supported on the faces $\Delta''_{\Delta{\bar q}}$ for finite $\Delta''$. Thus $\Delta''_{\Delta{\bar q}}$ in the formulae below make sense.
\item 
Type 2: cells in the form $\Delta'_{\Delta{\bar q}\setminus v}$ where $v$ is a vertex or a divisorial vector in $\{\Delta \bar q\}$ which is not in $\Delta'$.
\end{itemize}
Next we show that these two boundary types endowed with the framing correspond to the $d''$ and $d'$ differentials in the SI complex, respectively, see Figure \ref{fig:h_2(1)}.

 \begin{figure}
    \centering
    \includegraphics[width=4.5in]{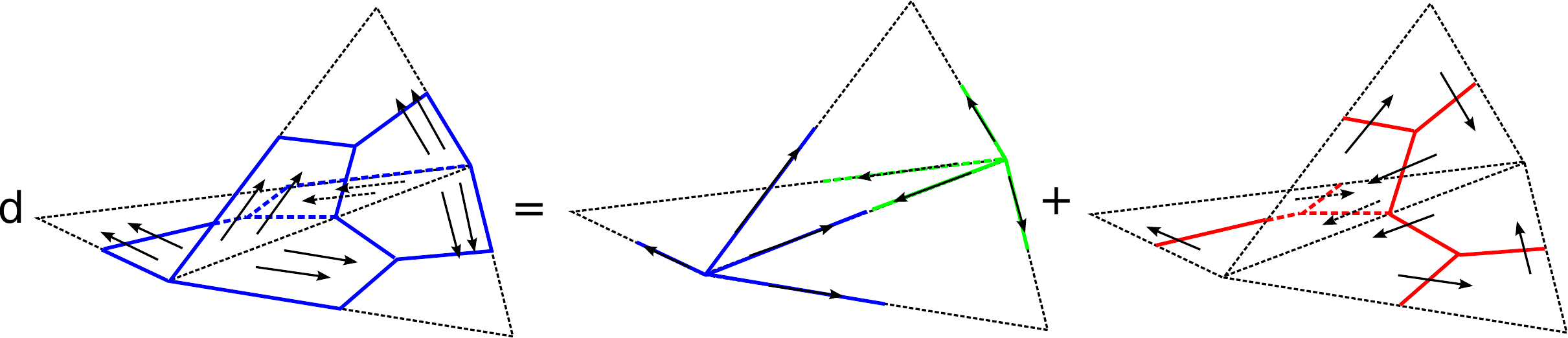}
    \caption{$d=d'+d'': H_2(1)\to H_2(0)\oplus H_0(2)[-1]$. (Framing coefficient vectors are not to scale).}
    \label{fig:h_2(1)}
  \end{figure}

Boundary of type 1: 
\begin{multline*}
\sum_{\Delta', \bar q} \ \sum_{q\in \bar q} (\rho_{\bar q} \Vol_{\Delta' \bar q})  \hat{\{\Delta' q\}}_{\Delta{\bar q}} 
+ \sum_{\bar q}\sum_{\Delta'\prec_1 \Delta''\prec \Delta}  (\rho_{\bar q} \Vol_{\Delta'\bar q}) \hat\Delta''_{\Delta{\bar q}}\\
= \sum_{q\in \Link_1^0(\Delta)}
\left( \sum_
{
\begin{subarray}{c}
\Delta''\prec \Delta q,\  \Delta''\not\prec\Delta\\
\bar r \in \Link_{l-1}(\Delta q)
\end{subarray}
} 
(\rho_{q \bar r} \Vol_{\Delta''\bar r}) \hat\Delta''_{\Delta{q \bar r}} 
+ \sum_
{
\begin{subarray}{c}
\Delta''\prec \Delta\\
\bar r \in \Link_{l-1}(\Delta q)
\end{subarray}
} 
(\rho_{q \bar r}  \Vol_{\Delta''\bar r})\right) 
\hat\Delta''_{\Delta{q \bar r}}\\
=\sum_{q \in \Link_1^0(\Delta)}  \sum_
{
\begin{subarray}{c}
\Delta''\prec \Delta q\\
\bar r \in \Link_{l-1}(\Delta q)
\end{subarray}
} 
(\rho_{q \bar r}  \Vol_{\Delta''\bar r} ) \hat\Delta''_{\Delta{q \bar r}} .
\end{multline*}
Here in the second summand we used the identity \eqref{eq:simplex} for the pair $\Delta' \prec \Delta''\bar q$. From \eqref{eq:d''} one can easily see that this coincides with $(d''c)[-r-1]$. 

Boundary of type 2:
\begin{multline*}
 \sum_{q\in \Link_1(\Delta)} \sum_{
 \begin{subarray}{c}
 \Delta'\prec \Delta\\
  \bar r \in \Link_{l-1}(\Delta)
 \end{subarray}
 }
(  \rho_{q \bar r} \Vol_{\tau q \bar r}) \hat \Delta'_{\Delta{\bar r}} 
+ \sum_{v\in\Delta }  \sum_
{
 \begin{subarray}{c}
 \Delta'\prec \Delta\setminus v\\
  \bar q \in \Link_{l}(\Delta)
 \end{subarray}
} 
(\rho_{\bar q} \Vol_{\Delta' \bar q}) \hat\Delta' _{\Delta{\bar q}\setminus v} \\
= \sum_
{
 \begin{subarray}{c}
 v\in\Delta\\
 \Delta' \prec \Delta\setminus v \\
 \end{subarray}
} 
\left( \sum_{\bar r \in \Link_{l-1}(\Delta)}
(\rho_{v \bar r} \Vol_{\tau v \bar r}) \hat \Delta'_{\Delta{\bar r}} 
+ \sum_{\bar q \in \Link_{l}(\Delta)} 
(\rho_{\bar q}  \Vol_{\tau\bar q}) \hat \Delta'_{\Delta{\bar q}\setminus v} \right).
\end{multline*}
 Here in the first summand we used the identity \eqref{eq:vol_balance} for each $\Delta', \bar r$ with the sign compensated by the orientation of $\hat\Delta'_{\Delta{\bar r}}$ and the choice of $ \Vol_{\tau v \bar r}$. Taking the sum of \eqref{eq:d'} over all vertices $v\in \Delta$ we easily identify the last expression with $(d'c)[-r]$.
\end{proof}

Combining the above proposition with Theorem \ref{theorem:main} we can conclude that the konstruktor complex can be used to calculate the tropical homology groups $H_q(X; \F_p)$:

\begin{corollary}
The inclusion of the konstruktor $\oplus_r K_{p-r}(\bullet -p+2r)[-r]$  into the complex  $C^{bar}_\bullet(X;\F_p)$ is a quasi-isomorphism for each $p$.
\end{corollary}

Finally, since all infinite cells in the konstruktor chains have coefficients divisible by the divisorial directions we can use the explicit description \eqref{eq:wave_action} of the eigenwave action on it. Then unveiling the konstruktor definition we arrive at the following.
\begin{proposition}\label{konstruktor_action}
For any $c\in H_{2l}(\Delta)$ one has $\phi \cap (c[-r])=c[-r-1]$.
\end{proposition}

Now we can combine all above observations to prove the claimed isomorphism 
$$\phi^{q-p}: H_q(X; \F_p) \to H_{p}(X; \F_{q}).$$

\begin{proof}[Proof of Theorem \ref{theorem:isomorphism}]
The cap product action of the eigenwave $\phi^{q-p}$ on the homology $H_q(X, \F_p)$ can be induced from its action on the konstruktor, which is a simplicial chain subcomplex. But it agrees there with the classical action of the monodromy $\nu^{q-p}$ on the $E_1$ term of the SI spectral sequence. On the other hand it is well known that the $\nu^{q-p}$ induces an isomorphism on the associated graded pieces with respect to the monodromy weight filtration on $H_{p+q}(X_t)$, which are calculated on the $E_2$ term of the SI spectral sequence. 
\end{proof}

\begin{acknowledgement}
%
We are grateful to Ilia Itenberg for numerous useful discussions. We also wish to thank the referee for pointing out several mistakes and suggesting many exposition improvements. Finally we would like to thank the Max-Planck-Institut-f\"ur-Mathematik for its hospitality
during the special program ``Tropical Geometry and Topology".
\end{acknowledgement}

\end{document}